\documentclass[11pt,reqno]{amsart}

 \usepackage{amsfonts,amsmath,amscd,amssymb,amsbsy,amsthm,amstext,amsopn,amsxtra}
 \usepackage{color,fullpage,mathrsfs}
 \usepackage{subfigure}
 \usepackage{verbatim} 
 \usepackage{stmaryrd}
 \usepackage{extarrows}
 \usepackage[all]{xy}
 \usepackage{bm}   
 \usepackage{hyperref}
 \usepackage{enumitem}
 
\usepackage{cancel}
\usepackage{color}
\usepackage{tikz}
 \usepackage{tikz-cd}

 \numberwithin{equation}{section}
 \hypersetup{colorlinks,linkcolor=[rgb]{0,0,1},citecolor=[rgb]{1,0,0}}

\newtheorem{theorem}{Theorem}[section]
\newtheorem{lemma}[theorem]{Lemma}
\newtheorem{proposition}[theorem]{Proposition}
\newtheorem{corollary}[theorem]{Corollary}

\theoremstyle{definition}

\newtheorem{example}[theorem]{Example}
\newtheorem{examples}[theorem]{Examples}

\newtheorem{remark}[theorem]{Remark}

\newcommand{\one}{\ensuremath{(\mathrm{i})}}
\newcommand{\two}{\ensuremath{(\mathrm{ii})}}
\newcommand{\three}{\ensuremath{(\mathrm{iii})}}

\newcommand{\aLus}{\ensuremath{(\mathrm{a})}}
\newcommand{\bLus}{\ensuremath{(\mathrm{b})}}

\newcommand{\CC}{\ensuremath{\mathbb{C}}} 
\newcommand{\kk}{\ensuremath{\Bbbk}} 
\newcommand{\NN}{\ensuremath{\mathbb{N}}} 
\newcommand{\QQ}{\ensuremath{\mathbb{Q}}}

\newcommand{\ZZ}{\ensuremath{\mathbb{Z}}}

\newcommand{\End}{\operatorname{End}}

\newcommand{\git}{\ensuremath{\operatorname{\!/\!\!/\!}}}
\newcommand{\GL}{\operatorname{GL}}
\newcommand{\head}{\operatorname{h}}
\newcommand{\Hilb}{\operatorname{Hilb}}
\newcommand{\Hom}{\operatorname{Hom}}
\newcommand{\id}{\operatorname{id}}

\newcommand{\Mat}{\operatorname{Mat}}
 
\newcommand{\McKay}{\ensuremath{\overline{Q}}}
\newcommand{\PiMcKay}{\operatorname{\overline{\Pi}}}

\newcommand{\Proj}{\operatorname{Proj}}
\newcommand{\Quot}{\operatorname{Quot}}
\newcommand{\rank}{\operatorname{rk}}
 
\newcommand{\Rep}{\operatorname{Rep}}
\newcommand{\Seshadri}{\operatorname{S}}
\newcommand{\SL}{\operatorname{SL}}

\newcommand{\Spec}{\operatorname{Spec}}
\newcommand{\Sym}{\operatorname{Sym}}
\newcommand{\supp}{\operatorname{supp}}
\newcommand{\tail}{\operatorname{t}}
\newcommand{\tr}{\operatorname{tr}}

\newcommand{\xpij}[2]{\operatorname{x}_{#1,#2}}
\newcommand{\xaij}[2]{\operatorname{x}_{#1,#2}}

\newcommand{\pij}{\operatorname{x}_{p,ij}}

\newcommand{\aij}{\operatorname{x}_{a,ij}}

\newcommand{\plij}{\operatorname{x}_{p_l,ij}}

\title{Hilbert schemes of points on canonical surfaces}

\author{Alastair Craw} 
\author{Ryo Yamagishi} 

\address{Department of Mathematical Sciences, 
University of Bath, 
Claverton Down, 
Bath BA2 7AY, 
UK.}
\email{a.craw@bath.ac.uk}

\address{Department of Mathematics, Faculty of Science, Kyushu University, 744, Motooka, Nishi-ku, Fukuoka, Japan 819-0395.}
\email{yamagishi@kyushu-u.ac.jp}

\begin{document}

\begin{abstract}
 For $n\geq 1$, we investigate the Hilbert scheme of $n$-points on a surface $S$ with canonical singularities. We generalise the well-known theorem of Fogarty~\cite{Fogarty68} by showing that the underlying reduced subscheme of $\Hilb^n(S)$ is a normal variety of dimension $2n$ with canonical singularities, and for $n\leq 7$, we show that $\Hilb^n(S)$ is reduced. When $S$ has symplectic singularities over $\CC$, we show that the underlying reduced subscheme of $\Hilb^n(S)$ also has symplectic singularities, thereby generalising a result of Beauville~\cite{Beauville83}. Our results build on work of the first author with Gyenge, Gammelgaard and Szendr\H{o}i~\cite{CGGS21} that sought to identify the underlying reduced subscheme of the Hilbert scheme of $n$-points on a Kleinian singularity with a Nakajima quiver variety.  
  \end{abstract}

 \maketitle
\tableofcontents

\section{Introduction}
 Let $n\geq 1$ and let $S$ be a connected quasi-projective surface over an algebraically closed field $\kk$
 of characteristic zero. A celebrated result of Fogarty~\cite{Fogarty68} establishes that when $S$ is nonsingular, then the Hilbert scheme of $n$-points on $S$, denoted $\Hilb^{n}(S)$, is a nonsingular, connected variety over $\kk$ of dimension $2n$.  When $S$ is singular, an example for which $\Hilb^{n}(S)$ 
 is reducible was given by Mir\'{o}-Roig--Pons-Llopis~\cite{MP13}. However, for a surface $S$ with canonical singularities, also known as Du Val, Kleinian or ADE singularities,  Zheng~\cite{Zheng23} proved that $\Hilb^{n}(S)$
 is an irreducible scheme of dimension $2n$, while Jelisiejew~\cite[Problem~XIII]{Jelisiejew23} asked whether $\Hilb^n(S)$
 is reduced. 
 
 Our main result establishes the appropriate generalisation of Fogarty's theorem for surfaces with canonical singularities, at least if we give $\Hilb^n(S)$ 
 the underlying reduced scheme structure:
 
  \begin{theorem}
  \label{thm:HilbnSintro}
 Let $S$ be a connected, quasi-projective surface with canonical singularities. For $n\geq 1$, the underlying reduced subscheme of $\Hilb^{n}(S)$
 is a normal variety of dimension $2n$ with canonical singularities.
\end{theorem}

 We are unable to prove that $\Hilb^n(S)$ 
 is reduced in general, but we do provide an affirmative answer to Jelisiejew's question for small values of $n$:
 
 \begin{theorem}
    \label{thm:nleq7intro}
    For $n\leq 7$, the scheme $\Hilb^n(S)$
    is reduced, so  Theorem~\ref{thm:HilbnSintro} holds for the natural scheme structure on $\Hilb^n(S)$. In particular, $\Hilb^n(S)$ is normal for $n\leq 7$. 
\end{theorem}

 In addition, recall~\cite{Beauville00} that a normal variety $X$ over the field $\mathbb{C}$ of complex numbers has \emph{symplectic singularities} if the nonsingular locus of $X$ admits an everywhere non-degenerate closed 2-form $\omega$ such that for any resolution of singularities $\pi\colon Y\to X$, the pullback $\pi^*(\omega)$ extends to a regular 2-form on $Y$. Note that $\pi$ is a \emph{symplectic resolution} if $\pi^*(\omega)$ extends to an everywhere non-degenerate 2-form on $Y$. Our approach to Theorem~\ref{thm:HilbnSintro} allows us to establish the following result:

  \begin{corollary}
  \label{cor:Beauville}
 Let $S$ be a connected quasi-projective surface over $\mathbb{C}$ with symplectic singularities. For  $n\geq 1$, the underlying reduced subscheme of $\Hilb^{n}(S)$
 has symplectic singularities.
   \end{corollary}
       
 When $S$ is non-singular, then $\Hilb^{n}(S)$ is also non-singular, in which case Corollary~\ref{cor:Beauville} reduces to the well-known result of Beauville~\cite{Beauville83} that $\Hilb^{n}(S)$ 
 is a holomorphic symplectic variety. 

 \subsection{Nakajima quiver varieties}
 Recall that the \emph{good} (or \emph{principal}) \emph{component} of $\Hilb^n(S)$ 
 is the scheme-theoretic closure of the locus parametrising reduced subschemes of length $n$ in $S$. We deduce Theorem~\ref{thm:HilbnSintro} by establishing the following result when $S$ is a Kleinian singularity:

  \begin{theorem}
   \label{thm:Hilbnintro}
   Let $\Gamma\subset\SL(2,\kk)$ be a finite subgroup. For $n\geq 1$, the good component of $\Hilb^{n}(\mathbb{A}^2/\Gamma)$ is isomorphic to a Nakajima quiver variety $\mathfrak{M}_{\theta_0}$ for a special stability parameter $\theta_0$. Moreover, when working over $\mathbb{C}$, it has symplectic singularities
   and admits a unique projective, symplectic resolution.
   \end{theorem}

  Since $\Hilb^n(\mathbb{A}^2/\Gamma)$ is irreducible, the conclusions of Theorem~\ref{thm:Hilbnintro} actually hold for the underlying reduced subscheme of $\Hilb^{n}(\mathbb{A}^2/\Gamma)$. A similar statement was published in work of the first author with Gammelgaard, Gyenge and Szendr\H{o}i~\cite{CGGS21}, though Yehao Zhou pointed out a gap in the proof of \cite[Lemma~4.3]{CGGS21}. Here, we adopt a more algebraic approach to Theorem~\ref{thm:Hilbnintro} that enables us to prove a similar result for many more stability conditions, and we go further by establishing that $\Hilb^n(\mathbb{A}^2/\Gamma)$ is in fact reduced for $n\leq 7$. In addition, we extend these results to the global situation considered in Theorems~\ref{thm:HilbnSintro}-\ref{thm:nleq7intro} and Corollary~\ref{cor:Beauville}.
   
   To construct the quiver varieties of interest, consider the extended Dynkin graph of type ADE associated to $\Gamma$ by McKay~\cite{McKay80}, and add a framing node $\infty$ and a single edge joining $\infty$ to the extending node $0$ of the graph. The doubled quiver $Q$ of this graph has vertex set $Q_0=\{\infty, 0, 1, \dots, r\}$, and we write $\Pi$ for the preprojective algebra of $Q$. For $n\geq 1$, define the vector 
   \[
   v = (1,n\delta)\in \NN^{Q_0}=\NN\times \NN^{r+1},
   \]
 where $\delta\in \NN^{r+1}$ is the minimal imaginary root of the ADE affine root system. Nakajima quiver varieties are moduli spaces $\mathfrak{M}_\theta:=\mathfrak{M}_\theta(\Pi,v)$ of $\theta$-semistable $\Pi$-modules of dimension vector $v$, where $\theta = (\theta_i)_{i\in Q_0}$ is a stability condition. Work of Gan--Ginzburg~\cite{GanGinzburg06} implies that the scheme $\mathfrak{M}_\theta(\Pi,v)$ is actually reduced, in which case $\mathfrak{M}_\theta$ is a normal variety by work of Bellamy--Schedler~\cite{BS21}. This normality statement for quiver varieties, when combined with the isomorphism from Theorem~\ref{thm:Hilbnintro}, provides the essential ingredient in the normality statement from Theorem~\ref{thm:HilbnSintro}. 
  
 It is well-known that the affine quiver variety satisfies $\mathfrak{M}_0\cong \Sym^n(\mathbb{A}^2/\Gamma)$, while work carried out independently by Haiman, Kuznetsov~\cite{Kuznetsov07} and Nakajima produced a generic stability condition $\theta$ such that $\mathfrak{M}_\theta\cong \Hilb^{n}(S)$,
 where $S\to \mathbb{A}^2/\Gamma$ is the minimal resolution. The birational geometry of these quiver varieties was described completely by Bellamy--Craw~\cite{BC20}, where the GIT wall-and-chamber structure of a closed cone $F\subset \Theta_v$ is identified with the decomposition of the movable cone of $\mathfrak{M}_\theta$ into Mori chambers. In particular, every  projective, crepant partial resolution of $\Sym^n(\mathbb{A}^2/\Gamma)$ is isomorphic to $\mathfrak{M}_\theta$ for some $\theta\in F$. The stability condition $\theta_0$ from Theorem~\ref{thm:Hilbnintro} lies in an extremal ray of the cone $F$, so the induced morphism $\mathfrak{M}_{\theta_0}\to \mathfrak{M}_0$ obtained by variation of GIT quotient is of relative Picard rank one. This morphism fits into a commutative diagram 
  \begin{equation}
 \label{eqn:HilbertChow}
 \begin{tikzcd}
 \mathfrak{M}_{\theta_0} \ar[r,hook]\ar[d] & \Hilb^n(\mathbb{A}^2/\Gamma)\ar[d]\\  \mathfrak{M}_0\ar[r,"\cong"] & \Sym^n(\mathbb{A}^2/\Gamma),
  \end{tikzcd}
  \end{equation}
  where the right-hand vertical map is the Hilbert--Chow morphism and the closed immersion along the top identifies $\mathfrak{M}_{\theta_0}$ with the good component of the Hilbert scheme.

  \subsection{Sketch of the proof}
 The isomorphism from Theorem~\ref{thm:Hilbnintro} is obtained as the composition of three morphisms. 

  First, let $A$ be the algebra obtained as the quotient of the preprojective algebra $\Pi$ by the ideal generated by the unique arrow in $Q$ with head at the framing vertex $\infty$. Since $A$ is a quiver algebra, moduli spaces of $A$-modules of dimension vector $v$ are obtained as GIT quotients $\Rep(A,v)\git_\theta \GL(v)$ for some stability condition $\theta\in \Theta_v$, where $\GL(v)$ is a product of general linear groups acting by conjugation on the affine scheme $\Rep(A,v)$ parametrising $A$-modules of dimension vector $v$. We show that the natural closed immersion \begin{equation}
  \label{eqn:quivermoduliGIT}
\Rep(A,v)\git_\theta \GL(v)\longrightarrow\mathfrak{M}_\theta
\end{equation}
 is an isomorphism for all $\theta\in F$ (see Proposition~\ref{prop:AtoPi}), thereby generalising an observation for the Hilbert scheme of $n$-points in the affine plane by Nakajima~\cite[Proposition~2.7]{NakajimaHilbert}. 
 
 Secondly, we study quiver moduli spaces arising from a collection of subalgebras of $A$ indexed by nonempty subsets $J\subseteq \{0,1,\dots, r\}$. Following work of the first author with Gyenge--Gammelgaard--Szendr\H{o}i~\cite{CGGS21}, define $A_J$ to be the subalgebra of $A$ spanned by the classes of paths with head and tail in the set of vertices $\{\infty\}\cup J$, and let 
 $v_J\in \mathbb{N}\times \mathbb{N}^J$ be the image of $v$ under the projection  $\mathbb{N}^{Q_0}\to\mathbb{N}\times  \NN^J$. As above, a product of general linear groups $\GL(v_J)$ acts on the representation scheme $\Rep(A_J,v_J)$, and the GIT quotient $\Rep(A_J,v_J)\git_\eta \GL(v_J)$ for any stability condition $\eta$ is a moduli space of $A_J$-modules of dimension vector $v_J$. 
 
 To study the geometric link between the moduli spaces for the algebras $A_J$ and $A$, let $G_K$ denote the subgroup of $\GL(v)$ obtained as the product of general linear groups indexed by vertices in $K:=Q_0\setminus (\{\infty\}\cup J)$. Then $\GL(v_J)\cong \GL(v)/G_K$ acts on the $G_K$-invariant subalgebra $\kk[\Rep(A,v)]^{G_K}$, and a result of Lusztig~\cite{Lusztig98} provides generators for this algebra, generalising the well-known result of Le Bruyn--Procesi~\cite{LBP90} in the case $G_K=\GL(v)$.  Our key technical result leverages the inclusion of noncommutative algebras $A_J\hookrightarrow A$ to obtain a $\kk$-algebra epimorphism 
 \[
 \kk[\Rep(A_J,v_J)]\longrightarrow \kk[\Rep(A,v)]^{G_K}
 \]
 which in turn gives a $\GL(v_J)$-equivariant closed immersion from $\Rep(A,v)\git G_K$ into $\Rep(A_J,v_J)$. For each stability condition $\eta$ for $A_J$-modules of dimension vector $v_J$, we obtain a stability condition $\overline{\eta}$ for $A$-modules of dimension vector $v$ from $\eta$ by appending a zero for each vertex in $K$, and the above $\GL(v_J)$-equivariant closed immersion of affine schemes determines a closed immersion 
 \begin{equation}
     \label{eqn:epiintro}
\Rep(A,v)\git_{\overline{\eta}}\GL(v)\longrightarrow \Rep(A_J,v_J)\git_{\eta} \GL(v_J)
\end{equation}
 of GIT quotients. While our primary interest is when $J=\{0\}$, we nevertheless allow $J$ to be any nonempty subset of $\{0,1,\dots, r\}$ so that we can complete the quiver variety description of Quot schemes of Kleinian orbifolds~\cite{CGGS2} with the reduced scheme structure (see Section~\ref{sec:CGGS21}).
  
  Finally, we restrict attention to the special case $J=\{0\}$ to construct the third morphism required in the proof of Theorem~\ref{thm:Hilbnintro}. The stability condition $\eta=(-n,1)$ for $A_J$-modules of dimension vector $v_J$ determines the stability condition 
  \[
  \theta_0:=\overline{\eta}=(-n,1,0,0,\dots, 0)\in \Theta_v
  \]
  for $A$-modules of dimension vector $v$. The morphisms \eqref{eqn:quivermoduliGIT} and \eqref{eqn:epiintro} fit into a commutative diagram 
 \begin{equation}
\label{eqn:quivertoHilbintro}
\begin{tikzcd}
\mathfrak{M}_{\theta_0} \ar[r,hook] & \Hilb^{n}(\mathbb{A}^2/\Gamma)  \\
\Rep(A,v)\git_{\theta_0}\GL(v)\ar[r,hook,"\eqref{eqn:epiintro}"]\ar[u,"\eqref{eqn:quivermoduliGIT}"] & \Rep(A_J,v_J)\git_{(-n,1)} \GL(v_J)\ar[u,swap,"\cong"]
  \end{tikzcd}
  \end{equation}
  of schemes over $\Sym^n(\mathbb{A}^2/\Gamma)$, where the right-hand vertical map is the isomorphism constructed in \cite[Proposition~6.2]{CGGS21}. All morphisms in diagram \eqref{eqn:quivertoHilbintro} restrict to isomorphisms over the dense open subset $U\subset \Sym^n(\mathbb{A}^2/\Gamma)$ parametrising  $n$ distinct unordered points in $\mathbb{A}^2/\Gamma$, and it follows that the closed immersion along the top of diagram \eqref{eqn:quivertoHilbintro}  identifies the quiver variety $\mathfrak{M}_{\theta_0}$ with the good component of the Hilbert scheme $\Hilb^n(\mathbb{A}^2/\Gamma)$ as in Theorem~\ref{thm:Hilbnintro}. 
  
  The remaining geometric statements about this good component from Theorem~\ref{thm:Hilbnintro} were shown by Bellamy--Schedler~\cite{BS21}.  In addition, for $n\le7$, our proof that $\Hilb^n(\mathbb{A}^2/\Gamma)$ is actually reduced builds on recent work of Hu~\cite{Hu25} which shows that $\Hilb^n(\mathbb{A}^3)$ is Gorenstein.

 \subsection{On canonical surfaces}
To obtain Theorem~\ref{thm:HilbnSintro}, we study formal neighbourhoods of closed points in the Hilbert scheme of points on an arbitrary affine scheme of finite type over $\kk$. We show that every such formal neighbourhood depends only on Hilbert schemes of formal neighbourhoods of points in the scheme itself. In particular, for a surface $S$ with canonical singularities, the formal neighbourhood of any point of $\Hilb^n(S)$
is a product of Hilbert schemes of formal neighbourhoods of closed points in $S$, each of which is either nonsingular or is \'{e}tale-locally isomorphic to the singular point of $\mathbb{A}^2/\Gamma$ for some finite subgroup $\Gamma\subset \SL(2,\kk)$. The normality statement from Theorem~\ref{thm:Hilbnintro} allows us to conclude that the formal neighbourhood of each closed point in the underlying reduced subscheme of $\Hilb^n(S)$ 
is normal. Then $\Hilb^n(S)$ 
is irreducible because it is connected, so its underlying reduced subscheme is a variety over $\kk$. The reducedness statement in Theorem~\ref{thm:nleq7intro} follows directly from reducedness of $\Hilb^n(\mathbb{A}^2/\Gamma)$ for $n\leq 7$ (see Theorem~\ref{thm:reduced}).

  For a surface $S$ with symplectic singularities over $\CC$,  we establish Corollary~\ref{cor:Beauville} after analysing the Hilbert--Chow morphism from \eqref{eqn:HilbertChow}  via the stratification of $\Sym^n(S)$ by symplectic leaves.

\medskip

\noindent \textbf{Notation.} Throughout the paper, $\kk$ is an algebraically closed field of characteristic zero. Let $\Mat(m\times n)$ denote the $\kk$-vector space of $m\times n$ matrices with entries in $\kk$. We write nontrivial paths in $Q$ as products $p = a_\ell\cdots a_1$ of arrows from right to left, starting at vertex $\tail(p):=\tail(a_1)$ and ending at vertex $\head(p):=\head(a_{\ell})$. A cycle in $Q$ is a path whose head and tail vertices coincide.

\medskip

\noindent \textbf{Acknowledgements.} We are grateful to an anonymous referee for pointing out that Theorem~\ref{thm:Lusztig}, for which we supplied a proof in the original version of this paper, is in fact a theorem of Lusztig. Thanks also to Yehao Zhou for pointing out the gap in the proof of \cite[Lemma~4.3]{CGGS21}, and to S{\o}ren Gammelgaard, Joachim Jelisiejew, Hal Schenck and Ruth Wye for helpful conversations. The authors were supported by Research Project Grant RPG-2021-149 from The Leverhulme Trust.

\section{Invariants for quiver algebras}
We recall Lusztig's generalisation of the Le Bruyn--Procesi theorem that computes the invariant subring of the coordinate ring of the representation scheme. We also introduce the notation required for the study of algebras defined by a quiver with relations.

 \subsection{Quiver representations} 
 Let $Q$ be a finite, connected quiver with vertex set $Q_0$ and arrow set $Q_1$. For any vector $v=(v_i)\in \mathbb{N}^{Q_0}$, the representation space is the $\kk$-vector space 
 \begin{equation}
    \label{eqn:Repold}
\Rep(Q,v):=\bigoplus_{a\in Q_1} \Mat(v_{\head(a)}\times v_{\tail(a)})
\end{equation}
 parametrising representations of $Q$ of dimension vector $v$, where $\head(a)$ and $\tail(a)$ denote the vertices at the head and tail respectively of $a$. Let $\kk Q$ denote the path algebra of $Q$, and for $i\in Q_0$, write $e_i\in \kk Q$ for the trivial path at vertex $i$.

 For any nontrivial cycle $\gamma=a_\ell\cdots a_1$ in $Q$ obtained by traversing the arrows $a_1, a_2, \dots a_\ell$ in that order, the map sending $B\in \Rep(Q,v)$ to
 \[
 \tr_\gamma(B):= \tr(B_{a_\ell}\cdots B_{a_1})
 \]
 is a polynomial function $\tr_\gamma$ on $\Rep(Q,v)$ called the \emph{trace function} for the cycle $\gamma$. More generally, for any vertex $i$ and any linear combination $\gamma=\sum_k \lambda_k \gamma_k$ of nontrivial cycles passing through $i$, we obtain a trace function on $\Rep(Q,v)$ by setting $\tr_\gamma:= \sum_k \lambda_k \tr_{\gamma_k}$.

 To define contraction functions, let $p=a_\ell\cdots a_1$ be a nontrivial path in $Q$ from vertex $\tail(p)$ to vertex $\head(p)$. For $B=(B_a)\in \Rep(Q,v)$, the product satisfies $B_{a_\ell}\cdots B_{a_1}\in \Mat(v_{\head(p)}\times v_{\tail(p)})$. 
 For $1\leq i\leq v_{\head(p)}$ and $1\leq j\leq v_{\tail(p)}$, the map sending $B$ to the $(i,j)$-entry of this matrix, denoted
 \[
 \pij(B):= \big(B_{a_\ell}\cdots B_{a_1}\big)_{ij},
 \]
 defines the  \emph{$(i,j)$-contraction function} $\pij$ on $\Rep(Q,v)$. More generally, for $\tail, \head\in Q_0$ and any linear combination $g=\sum \lambda_l p_l$ of paths $p_l$, each with head and tail at $\head$ and $\tail$ respectively, extend linearly to obtain an $(i,j)$-contraction function $\xpij{g}{ij} := \sum_l \lambda_l \xpij{p_l}{ij}$ for  $1\leq i\leq v_{\head}$ and $1\leq j\leq v_{\tail}$.  The decomposition from \eqref{eqn:Repold} shows that the coordinate ring of the representation space is
 \begin{equation}
     \label{eqn:kRep}
 \kk\big[\!\Rep(Q,v)\big] = \bigotimes_{a\in Q_1} \kk\big[\aij\mid 1\leq i\leq v_{\head(a)}\text{ and }1\leq j\leq v_{\tail(a)}\big].
 \end{equation}
 
 The group $\GL(v):= \prod_{i\in Q_0} \GL(v_i,\kk)$ acts on $\Rep(Q,v)$ by conjugation, where $g=(g_i)\in \GL(v)$ and $B=(B_a)_{a\in Q_1}\in \Rep(Q,v)$ satisfy
 \begin{equation}
 \label{eqn:groupaction}
 g\cdot B = (g_{\head(a)} B_a g_{\tail(a)}^{-1})_{a\in Q_1}.
 \end{equation}
 The diagonal scalar subgroup $\kk^\times = \{(\lambda\id_{v_i})_{i\in Q_0} \mid \lambda\neq 0\}$ of $\GL(v)$ acts trivially, so a polynomial on $\Rep(Q,v)$ is $\GL(v)$-invariant if and only if it is invariant under the action of $\GL(v)/\kk^\times$. The trace function $\tr_\gamma$ of any nontrivial cycle in $Q$ is $\GL(v)$-invariant. More generally, for any subset $K\subseteq Q_0$, consider the subgroup 
 \begin{equation}
     \label{eqn:HK}
 G_K:= G_K(v)= \prod_{k\in K} \GL(v_k)
 \end{equation}
 of $\GL(v)$ indexed by the vertices in $K$. 
 If both the head and tail of a path $p$ lie in the set $Q_0\smallsetminus K$, then $\pij$ is a $G_K$-invariant function. 
 
 The next following result generalises the well-known theorem of Le Bruyn--Procesi~\cite{LBP90} that the subalgebra of $\GL(v)$-invariant polynomials $\kk[\Rep(Q,v)]^{\GL(v)}$ is generated by the trace functions $\tr_\gamma$ associated to cycles $\gamma$ in the quiver $Q$.
 
\begin{theorem}[Lusztig]
\label{thm:Lusztig}
 For any subset $K\subseteq Q_0$, the algebra $\kk[\Rep(Q,v)]^{G_K}$ is generated by the:
 \begin{enumerate}
 \item[\aLus] trace functions $\tr_\gamma$ associated to cycles $\gamma$ in $Q$ that traverse only arrows in $Q$ with head and tail in $K$; and
 \item[\bLus] contraction functions $\pij$ for paths $p$ in $Q$ with head and tail in $Q_0\setminus K$, where the indices satisfy $1\leq i\leq v_{\head(p)}$ and $1\leq j\leq v_{\tail(p)}$.
 \end{enumerate}
 \end{theorem}
\begin{proof}
 See \cite[Theorem~1.3]{Lusztig98}. We provide an independent proof in \cite[Theorem~2.2]{CY23}. 
\end{proof}

The description of $\kk[\Rep(Q,v)]^{G_K}$ from Theorem~\ref{thm:Lusztig} interpolates between the description of the invariant ring $\kk[\Rep(Q,v)]^{\GL(v)}$ from \cite{LBP90} when $K=Q_0$, and the description of the coordinate ring $\kk[\Rep(Q,v)]$ from \eqref{eqn:kRep} when $K=\varnothing$.

 \subsection{Quiver algebras}
 A \emph{relation} for a quiver $Q$ is a finite linear combination of paths in $Q$, each with the same head and tail. Thus, for each relation $g$, there are vertices $\head(g), \tail(g)\in Q_0$ such that $g$ lies in the $\kk$-vector subspace $e_{\head(g)} \kk Q e_{\tail(g)}$ spanned by paths with tail $\tail(g)$ and head $\head(g)$; explicitly, 
\begin{equation}
    \label{eqn:relation}
 g = \sum_{1\leq l\leq \ell} \lambda_l p_l,
\end{equation}
 where $\lambda_1, \dots,\lambda_\ell\in \kk$, and paths $p_1, \dots, p_\ell$ satisfy $\tail(p_l)=\tail(g)$ and $\head(p_l)=\head(g)$ for all $1\leq l\leq \ell$. A representation $B=(B_a)\in \Rep(Q,v)$ 
 is said to \emph{satisfy the relation \eqref{eqn:relation}} if
 \begin{equation}
     \label{eqn:matrixrelation}
 \sum_{1\leq l \leq \ell}\lambda_l B_{p_l}= 0.
  \end{equation}
 The matrix equation  \eqref{eqn:matrixrelation} encodes a collection of $v_{\tail(g)}v_{\head(g)}$ equations in the entries of the matrices in $B$. Thus, the representation $B$ satisfies the relation \eqref{eqn:relation} if and only if $B$, when regarded as a point in the affine space $\Rep(Q,v)$, lies in the subscheme cut out by the equations
 \begin{equation}
     \label{eqn:variablerelation}
 \xpij{g}{ij}=\sum_{1\leq l \leq \ell} \lambda_l \plij =0
  \end{equation}
 for all $1\leq i\leq v_{\head(g)}$ and $1\leq j\leq v_{\tail(g)}$.

 A \emph{quiver algebra} is a $\kk$-algebra $A$ for which there is a quiver $Q$ and a surjective map
 $\beta\colon \kk Q \rightarrow A
$ of $\kk$-algebras. For each $i\in Q_0$, we use the same notation $e_i\in A$ for the class in $A$ of the trivial path.  Every quiver algebra is a $Q_0$-bigraded algebra $A=\bigoplus_{\tail, \head\in Q_0} e_{\head} A e_{\tail}$, where $e_{\head} A e_{\tail}$ is the $\kk$-vector subspace of $A$ spanned by the classes of paths with head and tail at vertex $\head$ and $\tail$ respectively. Write the kernel of $\beta$ as the two-sided ideal $\langle g_1, \dots, g_m\rangle$, where each $g_i$ is a relation in $Q$. 
  
  The \emph{representation scheme} $\Rep(A,v)$ is the subscheme of $\Rep(Q,v)$ cut out by the ideal generated by all of the functions $\xpij{g_k}{ij}$ associated to the generators $g_1, \dots, g_m$ of $\ker(\beta)$; this does not depend on the choice of generators of $\ker(\beta)$.  More invariantly, if we write $\tau\colon \kk[\Rep(Q,v)]\rightarrow \kk[\Rep(A,v)]$ for the surjective $\kk$-algebra homomorphism given by restriction of functions, then we may write
 \begin{equation}
 \label{eqn:kertau}
 \ker(\tau) =  \big\langle \xpij{p}{ij} \mid p\in e_{\head}\ker(\beta)e_{\tail} \text{ for }\head, \tail\in Q_0, \; 1\leq i\leq v_{\head}, \; 1\leq j\leq v_{\tail}\big\rangle.
 \end{equation}
In fact, we can describe polynomial functions on the representation scheme $\Rep(A,v)$ more directly using the quiver presentation $\beta\colon \kk Q\to A$ as follows.

 \begin{lemma}
 \label{lem:restrictfunctions}
 Fix vertices $\head, \tail\in Q_0$ and let $f\in e_{\head} A e_{\tail}$. For $1\leq i\leq v_{\head}$, $1\leq j\leq v_{\tail}$, the function $\xpij{f}{ij}:= \xpij{p}{ij}\vert_{\Rep(A,v)}$ in $\kk[\Rep(A,v)]$ obtained by restriction is well-defined independent of the choice of lift $p\in \beta^{-1}(f)$. Thus, for any $p\in e_{\head}\kk Q e_{\tail}$, we have 
  \[
  \tau(\xpij{p}{ij})=\xpij{p}{ij}\vert_{\Rep(A,v)} = \xpij{\beta(p)}{ij}.
  \]
 Similarly, the trace function of each cycle $\gamma$ in $Q$ satisfies $\tau(\tr_\gamma)=\tr_\gamma\vert_{\Rep(A,v)} = \tr_{\beta(\gamma)}$. 
 \end{lemma}
 \begin{proof}
See \cite[Lemma~3.1]{CY23}.
\end{proof}
 
 More generally, since $\GL(v)$ is reductive and $\kk$ has characteristic zero, the restriction of $\tau$ to the $G_K$-invariant subalgebra defines a surjective $\kk$-algebra homomorphism 
 \begin{equation}
 \label{eqn:tauK}
 \tau_K\colon \kk[\Rep(Q,v)]^{G_K}\rightarrow \kk[\Rep(A,v)]^{G_K}.
 \end{equation}  
 The next result extends Lusztig's Theorem~\ref{thm:Lusztig} to quiver algebras and computes $\ker(\tau_K)$.
 
 \begin{theorem}
 \label{thm:RepAinvariants}
  For any subset $K\subseteq Q_0$, the $\kk$-algebra $\kk[\Rep(A,v)]^{G_K}$ is generated by the functions:
 \begin{enumerate}
 \item[\aLus] $\tr_{\beta(\gamma)}$ for cycles $\gamma$ in $Q$ that traverse only arrows in $Q$ with head and tail in $K$; and 
 \item[\bLus] $\xpij{\beta(p)}{ij}$ for paths $p$ in $Q$ with head and tail in $Q_0\setminus K$, where $1\leq i\leq v_{\head(p)}$ and $1\leq j\leq v_{\tail(p)}$.
 \end{enumerate}
 Moreover, $\ker(\tau_K)$ is generated by the functions $\tr_\gamma$, where $\gamma\in e_k \ker(\beta)e_k$ for some $k\in K$, and the functions $\xpij{p}{ij}$, where $p\in e_{\head}\ker(\beta)e_{\tail}$ for some $\head, \tail\in Q_0\setminus K$, with $1\leq i\leq v_{\head}$ and $1\leq j\leq v_{\tail}$.
 \end{theorem}
 \begin{proof}
See \cite[Theorem~1.2]{CY23}.
  \end{proof} 
 
\begin{remark}
\label{rem:finitelymany}
 Only finitely many of the generators listed in Theorem~\ref{thm:RepAinvariants} are required because the algebra $\kk[\Rep(A,v)]$ is finitely generated over $\kk$ and $G_K$ is reductive.
\end{remark}

\section{The preprojective algebra of an extended ADE graph}

We now recall and extends results on Nakajima quiver varieties for both the McKay quiver and a framed version of the McKay quiver associated to the finite subgroup $\Gamma$. 
 
 \subsection{The preprojective algebra of the McKay quiver}
 Let $\Gamma\subset \SL(2,\CC)$ be a finite subgroup and write $V$ for the given two-dimensional representation of $\Gamma$. List the isomorphism classes of irreducible representations of $\Gamma$ as $\{\rho_0, \rho_1, \dots, \rho_r\}$, where $\rho_0$ is the trivial representation. The McKay quiver $\McKay$ has vertex set $\{0,1,\dots, r\}$, and for $0\leq i,j\leq r$, there are $\dim \Hom_\Gamma (\rho_i, \rho_j\otimes V)$ arrows from $i$ to $j$.  It follows from McKay~\cite{McKay80} that $\McKay$ is the doubled quiver of the extended Dynkin diagram of type ADE associated to $\Gamma$, that is, $\McKay$ is obtained from the Dynkin diagram by replacing each edge by two arrows, one for each orientation of the edge. For each arrow $a\in \McKay_1$, let $a^*\in \McKay_1$ denote the arrow corresponding to the same edge with opposite orientation. Choose a sign function $\epsilon\colon \McKay_1\to \{\pm 1\}$ such that $\epsilon(a)\neq \epsilon(a^*)$ for each arrow $a$. Let $\kk\McKay$ denote the path algebra of the McKay quiver. The \emph{preprojective algebra} of the McKay quiver is defined to be 
 \begin{equation}
     \label{eqn:PiGammaquotient}
 \PiMcKay:= \kk \McKay/\left\langle \textstyle{ \sum_{\{a\in \McKay_1 \vert \head(a)=i\}} \epsilon(a) aa^*  \;\vert\; 0\leq i\leq r}\right\rangle.
 \end{equation}
 
   For $R_0:= \kk[x,y]^\Gamma$, the $\Gamma$-module $\kk[x,y]$ decomposes into isotypical components $\bigoplus_{0\leq i\leq r} R_i$, where $R_i$ is the reflexive $R_0$-module obtained as the sum of all $\Gamma$-submodules of $\kk[x,y]$ that are isomorphic to $\rho_i$. For $0\leq i\leq r$, write $e_i$ for the idempotent in $\PiMcKay$ for vertex $i$ in $\McKay$. The isomorphism of noncommutative $\kk$-algebras 
  \begin{equation}
  \label{eqn:PiEnd}
  \PiMcKay\cong \End_{R_0}\bigg( \bigoplus_{0\leq i\leq r} R_i\bigg)
  \end{equation}
  (see, for example, \cite[Lemma~3.1]{Craw21}) restricts to the isomorphism of $\kk$-algebras 
  \begin{equation}
      \label{eqn:ADEhypersurface}
  e_0\PiMcKay e_0\cong R_0 = \kk[x,y]^\Gamma\cong   \kk[z_1, z_2, z_3]/(f_\Gamma)
  \end{equation}
  appearing in the work of Crawley-Boevey--Holland~\cite[Theorem~8.10]{CBHolland}, where $f_\Gamma\in \kk[z_1, z_2, z_3]$ is the defining polynomial for the Kleinian singularity $\mathbb{A}^2/\Gamma\cong (f_\Gamma=0)\subset \mathbb{A}^3$.  

 The vector $\delta=(\delta_i)\in \NN^{r+1}$ defined by $\delta_i=\dim(\rho_i)$ for $0\leq i\leq r$ is the minimal imaginary root of the affine root system of type ADE. For each positive integer $n\geq 1$, the quotient description of $\PiMcKay$ from \eqref{eqn:PiGammaquotient} determines the representation scheme $\Rep(\PiMcKay,n\delta)$ as in \eqref{eqn:kertau} that parametrises $\PiMcKay$-modules of dimension vector $n\delta$. The group $\GL(n\delta):=\prod_{0\leq i\leq r} \GL(n\dim \rho_i)$ acts on $\Rep(\PiMcKay,n\delta)$ by conjugation, and the affine quotient is 
 \[
 \Rep(\PiMcKay,n\delta)\git \GL(n\delta):=\Spec \kk[\Rep(\PiMcKay,n\delta)]^{\GL(n\delta)}.
 \]
 For $n=1$, Crawley-Boevey~\cite[Lemma~3,3]{CBdecomp} builds on work of Kronheimer~\cite{Kronheimer89} and his work with Holland~\cite[Theorem~2]{CBHolland} to show that the $\kk$-algebra
 \begin{equation}
     \label{eqn:Kroneimeraffine}
\kk[\Rep(\PiMcKay,\delta)]^{\GL(\delta)}\cong \kk[x,y]^\Gamma
 \end{equation}
 is  generated by trace functions of cycles $p_1, p_2, p_3$ in $\McKay$ through vertex $0$ corresponding to the chosen generators of the algebra from \eqref{eqn:ADEhypersurface}. In particular, $\Rep(\PiMcKay,\delta)\git \GL(\delta)\cong \mathbb{A}^2/\Gamma$.

  \subsection{The preprojective algebra of the framed McKay quiver}
 
   The \emph{framed McKay quiver} $Q$ is obtained from the McKay quiver $\McKay$ by adding a framing vertex, denoted $\infty$, and a pair of arrows, denoted $b\colon \infty\to 0$ and $b^*\colon 0\to \infty$. Equivalently, $Q$ is the doubled quiver of the graph obtained from the extended Dynkin diagram of type ADE associated to $\Gamma$ by adding a framing vertex $\infty$ and an edge joining $\infty$ to vertex $0$.  The vertex set of $Q$ is $Q_0=\{\infty, 0, 1, \dots, r\}$. As for the McKay quiver,  for each arrow $a\in Q_1$, write $a^*$ for the same edge of the framed, extended Dynkin graph with opposite orientation. Extend the sign function $\epsilon$ introduced on arrow set of the McKay quiver by setting $\epsilon(b)=1$ and $\epsilon(b^*)=-1$ to obtain a sign function $\epsilon\colon Q_1\to \{\pm 1\}$. 
   
   Let $\kk Q$ denote the path algebra of the framed McKay quiver. The \emph{preprojective algebra} $\Pi$ for the framed McKay quiver $Q$ is 
   \begin{equation}
     \label{eqn:Pi}
 \Pi:= \kk Q / \left\langle  \textstyle{\sum_{\{a\in Q_1 \vert \head(a)=i\}} \epsilon(a) aa^*  \mid i\in Q_0}\right\rangle.
 \end{equation}
 For $n\geq 1$, consider the vector $v = (1,n\delta)\in \NN^{Q_0}=\NN\times \NN^{r+1}$ which we also write as $v=\rho_\infty+n\delta$ for $\rho_\infty=(1,0)\in \mathbb{N}\times \mathbb{N}^{r+1}$. The quotient description of $\Pi$ from \eqref{eqn:Pi} determines the representation scheme $\Rep(\Pi,v)$ as in \eqref{eqn:kertau} that parametrises $\Pi$-modules of dimension vector $v$. The group 
 \[
 \GL(v):=\GL(1)\times \prod_{0\leq i\leq r} \GL(n\dim \rho_i)
 \]
  acts on $\Rep(\Pi,v)$ by conjugation, and we write
 \[
 \mathfrak{M}_0:=\Rep(\Pi,v)\git \GL(v):=\Spec \kk[\Rep(\Pi,v)]^{\GL(v)}
 \]
 for the affine quotient. The decomposition theorem of Crawley--Boevey~\cite[Theorem~1.1]{CBdecomp} is stated in terms of the underlying reduced subschemes of affine quotients. Here we draw key details out from the proofs in \cite{CBdecomp} to draw conclusions for the natural scheme structure on these quotients. 

\begin{proposition}
\label{prop:M0Sym^n}
  For $n\geq 1$, the scheme $\mathfrak{M}_0$ is reduced and irreducible, and there are isomorphisms 
  \begin{equation}
      \label{eqn:M0isoms}
  \mathfrak{M}_0\cong \Rep(\PiMcKay,n\delta)\git \GL(n\delta)\cong \Sym^n(\mathbb{A}^2/\Gamma)
 \end{equation}
  of affine varieties over $\kk$.  Moreover, the coordinate ring $\kk[\Rep(\Pi,v)]^{\GL(v)}$ of $\mathfrak{M}_0$ is generated as a $\kk$-algebra by the trace functions of cycles in the quiver $Q$ through vertex $0
    \in Q_0$.
\end{proposition}
\begin{proof}
 Gan--Ginzburg~\cite[Theorem~3.3.3(iii)]{GanGinzburg06} shows that the scheme $\Rep(\Pi,v)$ is reduced, and hence so is $\mathfrak{M}_0=\Rep(\Pi,v)\git \GL(v)$. Bellamy-Schedler~\cite[Theorem~1.2]{BS21} show that $\mathfrak{M}_0$ is irreducible. At this stage, we already obtain the isomorphism $\mathfrak{M}_0\cong \Sym^n(\mathbb{A}^2/\Gamma)$ from the decompostion theorem and additional work of Crawley--Boevey~\cite[Lemma~9.2]{CBmoment}, just as described in \cite[Lemma~4.5]{BC20}. However, we require more to obtain our assertion about generators of $\kk[\Rep(\Pi,v)]^{\GL(v)}$, and for this we revisit details from the proof of \cite[Theorem~1.1]{CBdecomp}. 

 Apply \cite[Lemma~5.4]{CBdecomp} by regarding $Q_0$ as the disjoint union $\{\infty\}\cup \{0,1,\dots, r\}$ and  $v=\rho_\infty+n\delta$, where  $\rho_\infty:=(1,0)\in \NN\times \NN^{r+1}$. Thus, the dimension vector of any composition factor of a $\Pi$-module of dimension vector $v$ is either $\rho_\infty$ or it has support in $\{0,1,\dots, r\}$. The proof of \cite[Lemma~5.1]{CBdecomp} then constructs a closed immersion 
 \begin{equation}
     \label{eqn:nu_1}
       \nu_1\colon 
  \Rep(\PiMcKay,n\delta)\git \GL(n\delta)\longrightarrow \mathfrak{M}_0
   \end{equation}
  (determined by the vanishing of maps for the arrows $b, b^*$ in $Q$; here, we omit a factor of $\Spec(\kk)$ for simplicity) and shows that it's a bijection on closed points, so $\nu_1$ induces an isomorphism of  underlying reduced subschemes. The scheme $\mathfrak{M}_0$ is reduced, so Lemma~\ref{lem:reducedness} implies that $\nu_1$ is an isomorphism of schemes. This gives the left-hand isomorphism in \eqref{eqn:M0isoms}. To construct the right-hand isomorphism in \eqref{eqn:M0isoms}, notice that the isomorphism \eqref{eqn:nu_1} shows that the scheme $\Rep(\PiMcKay,n\delta)\git \GL(n\delta)$ is reduced. In particular, both domain and codomain of the morphism 
  \begin{equation}
      \label{eqn:nu_2}
  \nu_2\colon \Sym^n(\mathbb{A}^2/\Gamma)\longrightarrow \Rep(\PiMcKay,n\delta)\git \GL(n\delta)
  \end{equation}
   constructed in the proof of \cite[Theorem~3.4]{CBdecomp} are reduced schemes. Crucially, the running assumption from \emph{loc.cit.} that \emph{all schemes are given the reduced scheme structure} is therefore automatic in our case, so $\nu_2$ is also an isomorphism of schemes by \cite[Theorem~3.4]{CBdecomp} as required. 
  
  To prove the final statement, the proof of \cite[Theorem~3.4]{CBdecomp} shows that the coordinate ring $\kk[\Rep(\PiMcKay,n\delta)]^{\GL(n\delta)}$ of $\Rep(\PiMcKay,n\delta)\git \GL(n\delta)$ is generated by the trace functions $\tr_p$ for $\Rep(\PiMcKay,n\delta)$ associated to cycles $p$ through vertex $0$ in $\McKay$, each of which takes the form $p=p_1^{r_1}p_2^{r_2}p_3^{r_3}$ for some $r_1, r_2, r_3\geq 0$; here, $p_1, p_2, p_3$ are the cycles in $\McKay$ through vertex $0$ whose trace functions generate the algebra \eqref{eqn:Kroneimeraffine} above.  It remains to pullback functions along the isomorphism $\nu_1^{-1}$ to conclude that $\kk[\Rep(\Pi,v)]^{\GL(v)}$ is generated by the trace functions $\tr_\gamma$ associated to cycles $\gamma$ in $Q$ through $0$ (recall from \eqref{eqn:Pi} that $b^*b=0$ in $\Pi$, while $bb^*$ can be expressed in terms of arrows  with heads and tails in $\{0,1,\dots, r\})$.  
\end{proof}

\begin{lemma}
\label{lem:reducedness}
 Let $f\colon X\to Y$ be a closed immersion of schemes with $Y$ reduced and Noetherian, such that $f$ induces an isomorphism on the underlying reduced schemes. Then $f$ is an isomorphism.
\end{lemma}
 \begin{proof}
   Let $g\colon Y\to X_{\operatorname{red}}$ be the inverse of the isomorphism on reduced subschemes induced by $f$, and let $h\colon X_{\operatorname{red}}\to X$ be the natural inclusion. Then $f\circ h\circ g\colon Y\to Y$ is a closed immersion. Since $Y$ is Noetherian, this is an isomorphism and hence so is $f$.
 \end{proof}
 
  We do not know whether the scheme $\Rep(\PiMcKay,n\delta)$ is reduced, but we record the following point which provides a $\Gamma$-analogue for Gan--Ginzburg's result on the commuting variety~\cite[Theorem~1.4]{GanGinzburg06}. 

 \begin{corollary}
 \label{cor:GG}
  The affine quotient $\Rep(\PiMcKay,n\delta)\git \GL(n\delta)$ is reduced for all $n\geq 1$. 
  \end{corollary}
 \begin{proof}
     This is immediate from the isomorphism of schemes $\nu_1$ constructed in Proposition~\ref{prop:M0Sym^n}.
 \end{proof}

 \subsection{The auxilliary algebra $A$}
 In what follows, it is convenient to work with the quotient algebra 
  \[
  A:= \Pi/(b^*),
  \]
   where $b^*$ is the unique arrow in the quiver $Q$ with head at vertex $\infty$. Note that $A$ is a quotient of the path algebra of the quiver $Q^*$ obtained by removing the arrow $b^*$ from $Q$. It follows that the representation scheme $\Rep(A,v)$ is the $\GL(v)$-invariant closed subscheme of $\Rep(\Pi,v)$ obtained by setting to zero the variables $\xpij{b^*}{i1}$ for $1\leq i\leq n$ indexed by arrow $b^*$.
   
 \begin{proposition}
 \label{prop:M0RepAv}
     There are isomorphisms 
     \[
     \mathfrak{M}_0\cong \Rep(A,v)\git \GL(v)\cong \Sym^n(\mathbb{A}^2/\Gamma),
    \]
    and the $\kk$-algebra $\kk[\Rep(A,v)]^{\GL(v)}$ is generated by trace functions of cycles through vertex $0$. 
 \end{proposition}
 \begin{proof}
  The quotient maps $\Pi\to A\to \PiMcKay\cong A/\langle b, e_\infty\rangle$ give rise to closed immersions
   \[
   \Rep(\PiMcKay,n\delta)\hookrightarrow \Rep(A,v)\hookrightarrow \Rep(\Pi,v)
   \]
  that are equivariant for the actions of $\GL(n\delta)$ and $\GL(v)$. The isomorphisms from \eqref{eqn:M0isoms} therefore factor via $\Rep(A,v)\git \GL(v)$, giving the isomorphisms $\mathfrak{M}_0\cong \Rep(A,v)\git \GL(v)\cong \Sym^n(\mathbb{A}^2/\Gamma)$. The proof of Proposition~\ref{prop:M0Sym^n} also shows that $\kk[\Rep(\PiMcKay,n\delta)]^{\GL(n\delta)}$ is generated by the trace functions of cycles in $\McKay$ through vertex 0, each of which traverses some compositions of the cycles $p_1, p_2, p_3$ in $\McKay$ through vertex $0$ whose trace functions generate the algebra from \eqref{eqn:Kroneimeraffine}. To see that $\kk[\Rep(A,v)]^{G(v)}$ is generated by the trace functions of the cycles in $Q^*$ through vertex $0$ corresponding to the same compositions of the cycles $p_1, p_2, p_3$, it remains to pull these functions back along the isomorphism $\kk[\Rep(A,v)]^{\GL(v)}\to \kk[\Rep(\PiMcKay,n\delta)]^{\GL(n\delta)}$ as in the proof of Proposition~\ref{prop:M0Sym^n}.
 \end{proof}

 \subsection{The auxilliary algebra suffices}
 The space of stability conditions for $\kk Q$-modules that have dimension vector $v$ is the rational vector space
  \begin{equation}
  \label{eqn:Thetav} 
 \Theta_v:=\big\{\theta\in \Hom_\ZZ(\ZZ^{Q_0},\QQ) \mid \theta(v)=0\big\},
  \end{equation}
 Each character $\chi_\theta$ of $\GL(v)/\kk^\times$ satisfies $\chi_\theta(g) = \prod_{k\in Q_0} \det(g_k)^{\theta_k}$ for $g\in \GL(v)/\kk^\times$, for some integer-valued $\theta\in \Theta_v$. For any such $\theta\in \Theta_v$, the corresponding \emph{Nakajima quiver variety} \cite{Nak94,Nak98} is defined to be the categorical quotient 
  \[
\mathfrak{M}_{\theta}:=\Rep(\Pi,v)\git_\theta \GL(v)=\Proj \bigoplus_{k\geq 0} \kk[\Rep(\Pi,v)]_{k\theta}
 \]
 of the locus of $\chi_\theta$-semistable points in $\Rep(\Pi,v)$ by the action of $\GL(v)$; here, $\kk[\Rep(\Pi,v)]_{k\theta}$ is the $\chi_{k\theta}$-semi-invariant slice of the coordinate ring of $\Rep(\Pi,v)$. In the case of interest to us, the natural scheme structure on $\mathfrak{M}_{\theta}$ is reduced. In fact, we can say the following:
   
 \begin{lemma}
 \label{lem:reduced}
   For $\theta\in \Theta_v$, the scheme $\mathfrak{M}_{\theta}$ is a normal variety. In addition, if the ground field is $\mathbb{C}$, then $\mathfrak{M}_{\theta}$ has symplectic singularities.
 \end{lemma}
 \begin{proof}
       Gan--Ginzburg~\cite[Theorem~3.3.3(iii)]{GanGinzburg06} show that $\Rep(\Pi,v)$ is reduced, and hence so is the categorical quotient $\mathfrak{M}_\theta$ of the $\chi_\theta$-semistable locus in $\Rep(\Pi,v)$. The result now follows by applying Bellamy-Schedler~\cite[Theorem~1.2]{BS21}.
     \end{proof}
  
 Following King~\cite{King94}, $\mathfrak{M}_\theta$ is the coarse moduli space of $\theta$-semistable $\Pi$-modules of dimension vector $v$ up to $\Seshadri$-equivalence. If each $\theta$-semistable $\Pi$-module of dimension vector $v$ is $\theta$-stable, then $\mathfrak{M}_\theta$ is the fine moduli space of $\theta$-stable $\Pi$-modules of dimension vector $v$ (up to isomorphism). When this is the case, $\mathfrak{M}_\theta$ carries a tautological locally-free sheaf 
 \[
 \mathcal{R}=\bigoplus_{i\in Q_0} \mathcal{R}_i
 \]
 and a $\kk$-algebra homomorphism $\phi\colon \Pi\to \End(\mathcal{R})$,  where $\mathcal{R}_\infty\cong \mathcal{O}_{\mathfrak{M}_\theta}$ and $\rank(\mathcal{R}_i)=n\dim(\rho_i)$ for $i\geq 0$. Here, $\mathcal{R}_i$ descends to $\mathfrak{M}_\theta$ from the $\GL(v)$-equivariant vector bundle $\varrho_i\otimes \mathcal{O}_{\Rep(\Pi,v)^{\theta}}$ on the $\theta$-stable locus in $\Rep(\Pi,v)$ defined as follows: equip the locally-free sheaf $\mathcal{O}_{\Rep(A,v)^{\theta}}^{\oplus v_i}$ with the action of $\GL(v)$ on each fibre given by the representation $\varrho_i\colon \GL(v)\to \GL(n\dim \rho_i)$ satisfying $\varrho(g)=g_i$. 
 
 The full GIT chamber decomposition of the vector space $\Theta_v$ is described explicitly in \cite{BC20}, but here we restrict attention to those characters that lie in the closure of the fundamental domain 
 \begin{equation}
     \label{eqn:F}
 F:= \big\{\theta\in \Theta_v \mid \theta(\delta)>0, \; \theta(\rho_i)>0 \text{ for }1\leq i\leq r\big\}
 \end{equation}
 for the action of the Namikawa--Weyl group. For each chamber $C\subset F$, variation of GIT quotient defines a projective, symplectic (hence crepant)  resolution of singularities 
 \[
 \pi_\theta\colon \mathfrak{M}_\theta\longrightarrow \mathfrak{M}_0\cong \Sym^n(\mathbb{A}^2/\Gamma)
 \]
 for $\theta\in C$. Conversely, every projective, symplectic resolution of $\Sym^n(\mathbb{A}^2/\Gamma)$ can be obtained in this way; indeed,  \cite[Theorem~1.2]{BC20} shows that the linearisation map identifies the GIT wall-and-chamber structure of $F$ with the decomposition of the movable cone of $\mathfrak{M}_\theta$ into Mori chambers. Two such resolutions have been well-studied: one chamber $C_+$ in $F$ defines 
  $\mathfrak{M}_{\theta_+}\cong n\Gamma\text{-}\Hilb(\mathbb{A}^2)$
  for $\theta_+\in C_+$, and a second chamber $C_-$ in $F$ defines $  \mathfrak{M}_{\theta_-}\cong \Hilb^{n}(S)
  $
  for $\theta_-\in C_-$, where $S\to \mathbb{A}^2/\Gamma$ is the minimal resolution (see \cite[Section~4.3]{BC20} and references therein).

The next result justifies our decision to work with the simpler algebra $A$ rather than $\Pi$.

 \begin{proposition}
 \label{prop:AtoPi}
    For $\theta\in F$, the closed immersion $\Rep(A,v)\hookrightarrow \Rep(\Pi,v)$ gives an isomorphism of schemes $\Rep(A,v)\git_\theta \GL(v)\to \mathfrak{M}_\theta$ over $\Sym^n(\mathbb{A}^2/\Gamma)$.
 \end{proposition}
  \begin{proof}
    First, let $\theta\in F$ be generic. Let $C$ denote the unique GIT chamber satisfying $\theta\in C$. If $C=C_+$, then $\mathfrak{M}_\theta$ can be described as an equivariant Hilbert scheme for $\theta\in C_+$, and the result is a restatement of \cite[Lemma~3.1]{CGGS21}. Otherwise, $C$ and $C_+$ are distinct chambers in $F$. For $\theta\in C$ and $\theta_+\in C_+$, the results from \cite[Corollary~6.3, Proposition~5.7]{BC20} imply that the diagram
       \begin{equation}
 \label{eqn:flopquivermoduli}
\begin{tikzcd}
\mathfrak{M}_\theta\ar[dr]\ar[rr,dashed] & & \mathfrak{M}_{\theta_+} \ar[dl]\\
  & \mathfrak{M}_0 &
  \end{tikzcd}
  \end{equation}
    obtained by variation of GIT quotient is a finite sequence of flops, and moreover, the tautological families $\mathcal{R}$ and $\mathcal{R}^+$ on  $\mathfrak{M}_\theta$ and $\mathfrak{M}_{\theta_+}$ respectively agree in codimension-one. Since the result is true for $\theta_+$, the tautological morphism $\mathcal{R}^+_0\to \mathcal{R}^+_\infty$ associated to the arrow $b^*\in Q_1$ is zero. Since $\mathcal{R}$ and $\mathcal{R}^+$ agree off a locus of codimension at least two, and since both $\mathfrak{M}_\theta$ and $\mathfrak{M}_{\theta_+}$ are nonsingular, the tautological morphism $\mathcal{R}_0\to \mathcal{R}_\infty$ over $\mathfrak{M}_\theta$ associated to $b^*$ is also zero. This proves the result for $\theta\in C$, and hence for any generic $\theta\in F$.

   Now let $\theta\in F$. Let $C\subset F$ be a chamber such that $\theta\in \overline{C}$, and fix $\zeta\in C$.  There is a commutative diagram of GIT quotients
 \begin{equation}
 \label{eqn:RepARepPiVGIT}
\begin{tikzcd}
\Rep(A,v)\git_{\zeta} G(v)\ar[r,"f_\zeta"]\ar[d,"h"] & \mathfrak{M}_{\zeta} \ar[d,"g"]\\
\Rep(A,v)\git_{\theta} G(v)\ar[r,hook,"f_\theta"]  & \mathfrak{M}_{\theta},
  \end{tikzcd}
  \end{equation}
  where $f_\zeta$ is an isomorphism, $f_\theta$ is the closed immersion induced by the $G(v)$-equivariant inclusion $\Rep(A,v)\hookrightarrow \Rep(\Pi,v)$, and where the vertical maps are induced by variation of GIT quotient. We claim first that $f_\theta$ is an isomorphism on the underlying reduced subschemes. For this, it is enough to show that $f_\theta$ is surjective on closed points.  Let $p\in \mathfrak{M}_\theta$ be a closed point. The morphism $g$ is surjective by \cite[Lemma~2.2]{CGGS21} and $f_\zeta$ is an isomorphism, so there exists $q\in \Rep(A,v)\git_\zeta G(v)$ such that $g(f_\zeta(q)) = p$. The diagram commutes, so $f_\theta(h(q))=p$ and hence $f_\theta$ induces an isomorphism of the underlying reduced schemes. The result follows from the general statement in Lemma~\ref{lem:reducedness}.
  \end{proof}

\section{Cornered algebras and quiver moduli spaces}
 \label{sec:cornering}
 This section studies representation schemes for a noncommutative algebra $A_J$ associated to any nonempty subset $J\subseteq \{0,1,\dots, r\}$. The main results in this paper - Theorems~\ref{thm:HilbnSintro}, \ref{thm:nleq7intro}, \ref{thm:Hilbnintro} and Corollary~\ref{cor:Beauville} - require only the case $J=\{0\}$, but the proofs do not become substantially more difficult for any subset $J$ containing $0$. The reader interested in the case where $0\not\in J$ should consult Appendix~\ref{sec:appendix} for a technical result that is required for Lemma~\ref{lem:nocyclesneededinA}.

 \subsection{The algebras $A_J$ obtained by cornering}
  We now recall a collection of noncommutative algebras introduced by Craw, Gammelgaard, Gyenge and Szendr\H{o}i~\cite[Section~3]{CGGS2}. 
  
  For the subquiver $Q^*$ of $Q$ obtained by deleting the arrow $b^*$ with head at $\infty$, and for each vertex $i\in Q_0^*$, write $e_i\in A$ for the image of the vertex idempotent under the epimorphism $\kk Q^*\to A$.  Let $J\subseteq \{0,1,\dots, r\}$ be any subset, write $e_J:= \sum_{j\in J} e_j$, and define
 \[
 A_J:=(e_\infty + e_J) A (e_\infty + e_J)
 \]
 to be the subalgebra of $A$ spanned by the classes of paths in $Q^*$ with head and tail in $\{\infty\}\cup J$. We write $\iota_J\colon A_J\hookrightarrow A$ for the inclusion, choosing to suppress the dependence on $J$ in the notation.
   
     \begin{lemma}
 \label{lem:AJA}
 For any non-empty subset $J\subseteq \{0,1,\dots, r\}$, there exists a finite connected quiver $Q_J^*$ and a surjective $\kk$-algebra homomorphism $\beta_J\colon \kk Q_J^*\to A_J$.
   \end{lemma}
 \begin{proof}
 The quiver $Q_J^*$ and the map $\beta_J$ are constructed in \cite[Proposition~3.3]{CGGS21}.
 \end{proof}

Define the vector $v_J \in \NN\times \NN^J$ to have first entry $1$ (for vertex $\infty$) and $v_j=n\dim(\rho_j)$ for $j\in J$.  The representation scheme $\Rep(A_J,v_J)$ is the subscheme of the affine space $\Rep(Q_J^*,v_J)$ cut out by the equations \eqref{eqn:variablerelation} associated to the relations \eqref{eqn:relation} generating $\ker(\beta_J)$. Write 
\begin{equation}
    \label{eqn:sigmaJ}
\sigma_J\colon \kk[\Rep(Q_J^*,v_J)]\longrightarrow \kk[\Rep(A_J,v_J)]
\end{equation}
 for the natural quotient map, where $\ker(\sigma_J)$ is given by \eqref{eqn:kertau} for the presentation $\beta_J\colon \kk Q_J^*\to A_J$.

 \begin{example}
 \label{exa:A_Jfor0}
  For $J=\{0\}$, the quiver $Q_J^*$ is shown in \cite[Figure~2]{CGGS21}: the vertex set is $\{\infty, 0\}$, and the arrow set is $\{b,a_1, a_2, a_3\}$, where each $a_i$ is a loop at vertex $0$. The ideal of relations is
  \[
   \ker(\beta_J)=\big(f(a_1, a_2, a_3), 
   a_1a_2-a_2a_1, 
   a_1a_3-a_3a_1, 
   a_2a_3-a_3a_2
   \big),
  \]
   where $f$ is the defining equation of the hypersurface $\mathbb{A}^2/\Gamma\subset \mathbb{A}^3$. In this case, the dimension vector is $v_J=(1,n)$, and the representation scheme is 
   \[
   \Rep(A_J,v_J) = \Mat(n\times 1)\times\left\{(B_1, B_2, B_3)\in \Mat(n\times n)^{\oplus 3} \mid 
   \begin{array}{c} f(B_1, B_2, B_3)=0 \\ 
   B_kB_l=B_l B_k \text{ for }1\leq k,l\leq 3
   \end{array}\right\}.
   \] 
     The ring $\kk[\Rep(A_J,v_J)]$ is the quotient of the polynomial ring 
$\kk[w_i, x_{ij}, y_{ij}, z_{ij}\mid 1\leq i,j\leq n]$ by the ideal generated by the $4n^2$ polynomials determined by the entries in the four matrix equations above. For example, when $\Gamma$ is of type $A_1$ and $n=2$, this ideal is 
\begin{equation}
\label{eqn:A1n2finalideal}
\left(\begin{array}{c} 
 x_{12}  y_{21} - x_{21}  y_{12},\;\; 
 x_{12}  z_{21} - x_{21}  z_{12},\;\;
   y_{12}  z_{21} - y_{21}  z_{12}, \\
 x_{11}  z_{11}+ x_{21}  z_{12}- y_{11}  y_{11}- y_{12}  y_{21},\;\; x_{12}  z_{11}+ x_{22}  z_{12}- y_{11}  y_{12}- y_{12}  y_{22} \\
 x_{11}  z_{21}+ x_{21}  z_{22} - y_{11}  y_{21}- y_{21}  y_{22}, \;\; x_{12}  z_{21}+ x_{22}  z_{22}- y_{12}  y_{21}- y_{22}  y_{22} \\
 x_{12}  y_{22}- x_{12}  y_{11}+ x_{11}  y_{12}- x_{22}  y_{12}, \;\;  x_{21}  y_{11}- x_{11}  y_{21}+ x_{22}  y_{21}- x_{21}  y_{22}\\
 x_{12}  z_{22} - x_{12}  z_{11}+ x_{11}  z_{12}- x_{22}  z_{12}, \;\; x_{21}  z_{11}- x_{11}  z_{21}+ x_{22}  z_{21}- x_{21}  z_{22}\\
 y_{12}  z_{11}- y_{11}  z_{12}+ y_{22}  z_{12}- y_{12}  z_{22}, \;\;  y_{21}  z_{22} - y_{21}  z_{11}+ y_{11}  z_{21}- y_{22}  z_{21}
      \end{array}\right)
    \end{equation}
One can check using Macaulay2~\cite{M2} that this ideal is prime, so the scheme $\Rep(A_J,v_j)$ is reduced and irreducible when $\Gamma$ is of type $A_1$, $n=2$ and $J=\{0\}$. 
\end{example}

 For any non-empty subset $J\subseteq \{0,1,\dots, r\}$, define $K:=\{0,1,\dots, r\}\setminus J$ so that
     \begin{equation}
       \label{eqn:partition}
       Q_0^*=\{\infty\}\sqcup J\sqcup K
   \end{equation}
  is a partition of the vertex set of $Q^*$.  The  subgroup $G_K$ of $\GL(v)$ from \eqref{eqn:HK} acts on $\Rep(Q^*,v)$, leaving $\Rep(A,v)$ invariant, and we obtain a surjective $\kk$-algebra homomorphism
  \begin{equation}
  \label{eqn:tauK*}
  \tau_K\colon \kk[\Rep(Q^*,v)]^{G_K}\longrightarrow \kk[\Rep(A,v)]^{G_K}
  \end{equation}
 as in \eqref{eqn:tauK}. Generators for the ideal $\ker(\tau_K)$ were computed in Theorem~\ref{thm:RepAinvariants}. 
 
 \begin{lemma}
 \label{lem:nocyclesneededinA}
 Let $J, K\subseteq Q_0^*$ satisfy \eqref{eqn:partition}. The algebra $\kk[\Rep(A,v)]^{G_K}$ is generated by functions $\xpij{\beta(p)}{ij}$ defined by paths $p$ in $Q^*$ with head and tail in $\{\infty\}\cup J$, where $1\leq i\leq v_{\head(p)}$ and $1\leq j\leq v_{\tail(p)}$. 
 \end{lemma}
 \begin{proof}
  By Theorem~\ref{thm:RepAinvariants}, we need only show that the trace functions $\tr_{\beta(\gamma)}\in \kk[\Rep(A,v)]^{\GL(v)}$ for cycles $\gamma$ in $Q^*$ that traverse only arrows with head and tail in $K$ are redundant. Proposition~\ref{prop:M0RepAv} shows that $\kk[\Rep(A,v)]^{\GL(v)}$ is generated by the trace functions of classes $\beta(p)$ defined by cycles in $Q^*$ through vertex 0. It follows that the trace functions of cycles in $Q^*$ that traverse only arrows with head and tail in $K$ all lie in the subalgebra generated by trace functions of cycles through $0$. If $0\in J$, then $\tr_{\beta(p)}=\sum_{1\leq i\leq n} \xpij{\beta(p)}{ii}$ expresses the trace function of every such class in terms of the contraction functions $\xpij{\beta(p)}{ij}$  given in the statement of the lemma, so the result holds. 
  
  Otherwise, $0\not\in J$. For any $j\in J$, Proposition~\ref{prop:CBresultsforvertexj} shows that $\kk[\Rep(\PiMcKay,n\delta)]^{\GL(n\delta)}$ is generated by trace functions of cycles through $j$. The final sentence in the proof of Proposition~\ref{prop:M0RepAv}, edited so that $0$ and $p_1, p_2, p_3$ are replaced by  vertex $j$ and the classes $q_1, q_2, q_3$ from \eqref{eqn:qell} respectively, shows that $\kk[\Rep(A,v)]^{\GL(v)}$ is generated by the trace functions of cycles through vertex $j$. The result now follows as in the previous paragraph: the trace functions of cycles in $Q^*$ that traverse only arrows with head and tail in $K$ all lie in the subalgebra generated by trace functions of cycles through $j$, and since $j\in J$, the equality $\tr_{\beta(p)}=\sum_{1\leq i\leq n} \xpij{\beta(p)}{ii}$ expresses the trace function of every such class in terms of the contraction functions $\xpij{\beta(p)}{ij}$ given in the statement of the lemma.
  \end{proof}

\subsection{The key algebra epimorphism}
  To state the key result, observe that the composition of the quiver presentation $\beta_J\colon Q_J^*\to A_J$ with the natural inclusion of $A_J$ into $A$ defines a $\kk$-algebra homomorphism $\iota_J\circ \beta_J\colon \kk Q_J^*\rightarrow A$.
  
 \begin{proposition}
 \label{prop:epi}
  Let $J, K\subseteq Q_0^*$ satisfy \eqref{eqn:partition}. There is a surjective homomorphism 
     \[
     \psi\colon \kk\big[\!\Rep(A_J,v_J)\big]\longrightarrow \kk\big[\!\Rep(A,v)\big]^{G_K}
     \]
  of commutative $\kk$-algebras. In particular, the algebra $\kk\big[\!\Rep(A,v)\big]^{G_K}$ is generated by the finite set $ \{\xpij{(\iota_J\circ\beta_J)(a)}{ij} \mid a \text{ is an arrow in }Q_J^*, 1\leq i\leq v_{\head(a)}, 1\leq j\leq v_{\tail(a)}\}$.
 \end{proposition}
 \begin{proof}
       We construct a map $\overline{\psi}$ that fits into a commutative diagram 
   \begin{equation}
\label{eqn:commutativealgebrahomosJ}
\begin{tikzcd}
  \kk[\Rep(Q_J^*,v_J)] \ar[d,swap,"\sigma_J"]\ar[dr,"\overline{\psi}"] & \\
  \kk[\Rep(A_J,v_J)] \ar[r,"\psi"] & \kk[\Rep(A,v)]^{G_K}
  \end{tikzcd}
  \end{equation}
of $\kk$-algebra homomorphisms, where $\sigma_J$ is the quotient map from \eqref{eqn:sigmaJ}. For each arrow $a$ in $Q_J^*$ and each index $1\leq i\leq v_{\head(a)}$ and $1\leq j\leq v_{\tail(a)}$, define
\[
\overline{\psi}(\xaij{a}{ij}) = \xaij{(\iota_J\circ\beta_J)(a)}{ij} \in \kk[\Rep(A,v)].
\]
 Since $\kk[\Rep(Q_J^*,v_J)]$ is a polynomial ring in the variables $\xpij{a}{ij}$, this extends to a $\kk$-algebra homomorphism $\overline{\psi}$. For each arrow $a$ in $Q_J^*$, the element $(\iota_J\circ\beta_J)(a)$ lies in $e_{\head(a)}A e_{\tail(a)}$ for $\tail(a), \head(a)\in \{\infty\}\cup J$, so each $ \xaij{(\iota_J\circ\beta_J)(a)}{ij} \in \kk[\Rep(A,v)]$ is $G_K$-invariant. To see that $\overline{\psi}$ factors via $\kk[\Rep(A_J,v_J)]$, recall from \eqref{eqn:kertau} shows that $\ker(\sigma_J)$ is generated by functions of the form $\xpij{p}{ij}$ for some $p\in e_{\head}\ker(\beta_J)e_{\tail}$ and $1\leq i\leq v_{\head}$, $1\leq j\leq v_{\tail}$, where $\head, \tail\in \{\infty\}\cup J$. Therefore $(\iota_J\circ\beta_J)(g)=0$ and hence $\overline{\psi}(\xaij{g}{ij}) = \xaij{0}{ij}$ is zero in $\kk[\Rep(A,v)]^{G_K}$. It follows that diagram \eqref{eqn:commutativealgebrahomosJ} exists and $\psi$ satisfies
    \[
 \psi(\xaij{\beta_J(a)}{ij}) = \xaij{(\iota_J\circ \beta_J)(a)}{ij} \in \kk[\Rep(A,v)]^{G_K}.
 \]
 
 For surjectivity, Lemma~\ref{lem:nocyclesneededinA} shows that $\kk[\Rep(A,v)]^{G_K}$ is generated by functions $\xpij{\beta(p)}{ij}$ defined by paths $p$ in $Q$ that don't traverse arrow $b^*$ and such that $\head(p), \tail(p)\in \{\infty\}\cup J$, where $1\leq i\leq v_{\head(p)}$ and $1\leq j\leq v_{\tail(p)}$.  Every such $\beta(p)$ lies in $e_{\head(p)} A e_{\tail(p)} = e_{\head(p)}\iota_J(A_J) e_{\tail(p)}$, so there exists $q\in e_{\head(p)}A_J e_{\tail(p)}$ such that $\iota_J(q) = \beta(p)$. Then $\psi(\xaij{q}{ij}) = \xaij{\iota_J(q)}{ij} = \xpij{\beta(p)}{ij}$, so $\psi$ is surjective.
 \end{proof}

  \begin{example}
 When $0\in J$, the group $G_K$ acts trivially on the variables $\xpij{b}{1j}$ for $1\leq j\leq n$ indexed by the unique arrow $b$ in $Q^*$ with tail at vertex $\infty$. It follows that
  \begin{equation}
      \label{eqn:algtensor}
  \kk[\Rep(A,v)]^{G_K}\cong \kk[\xpij{b}{11}, \dots, \xpij{b}{1n}]\otimes_\kk \kk[\Rep(\overline{\Pi},n\delta)]^{G_K}.
   \end{equation}
   In the special case from Example~\ref{exa:A_Jfor0} above when $\Gamma$ is of type $A_1$, $n=2$ and $K=\{1\}$, we describe in \cite[Example~4.4]{CY23} how elimination theory computes a presentation of $\kk[\Rep(\PiMcKay,2\delta)]^{G_K}$ as a quotient of the polynomial ring $\kk[x_{ij},y_{ij},z_{ij} \mid 1\leq i,j\leq 2]$ by the ideal with precisely the same generators as in \eqref{eqn:A1n2finalideal}. Take the tensor product of $\kk[\Rep(\PiMcKay,2\delta)]^{G_K}$ with $\kk[w_1, w_2]$ as in \eqref{eqn:algtensor} above, and compare with Example~\ref{exa:A_Jfor0}, to see explicitly in this case that $\kk[\Rep(A,v)]^{G_K}$ is isomorphic to $\kk[\Rep(A_J,v_J)]$ for $J=\{0\}$. In other words, the surjective $\kk$-algebra homomorphism $\psi$ from Proposition~\ref{prop:epi} is an isomorphism, so the affine quotient $\Rep(A,v)\git G_K$ is reduced and irreducible in this case. 
  \end{example}

 \subsection{A closed immersion of GIT quotients}  The geometric consequences of Proposition~\ref{prop:epi} can be seen by studying the action of the reductive group 
  \[
  \GL(v_J):=\GL(1)\times \prod_{j\in J} \GL(v_j)
  \]
  on $\Rep(Q_J^*,v_J)$ by conjugation, and this restricts to an action on the subscheme $\Rep(A_J,v_J)$. The partition \eqref{eqn:partition} implies that $\GL(v)=\GL(v_J)\times G_K$, so $\GL(v_J)\cong \GL(v)/G_K$ also acts on the affine quotient $\Rep(A,v)]\git G_K:=\Spec \kk[\Rep(A,v)]^{G_K}$.

 \begin{lemma}
 \label{lem:epi}
  The map $\psi$ induces a $\GL(v_J)$-equivariant closed immersion of affine schemes   
   \[
   \varphi=\Spec(\psi)\colon \Rep(A,v)\git G_K \longrightarrow \Rep(A_J,v_J).
   \]
    \end{lemma}
 \begin{proof}
  Each $B\in \Rep(A,v)\git G_K$ is the $G_K$-orbit of some tuple of matrices $(B_\alpha)_{\alpha\in Q_1^*}\in \Rep(A,v)$. For the quotient map $\beta_J\colon \kk Q_J^*\to A_J$ and for each arrow $a$ in $Q_J^*$, the class $\beta_J(a)\in A_J$ satisfies
  \[
 \xaij{\beta_J(a)}{ij}(\varphi(B)) = \psi(\xaij{\beta_J(a)}{ij})(B) = \xaij{(\iota_J\circ\beta_J)(a)}{ij}(B) = (B_{(\iota_J\circ\beta_J)(a)})_{ij}
\]
 because $\psi$ is the pullback along $\varphi$. More succinctly, if we write $\varphi(B)=(\varphi(B)_{\beta_J(a)})\in \Rep(A_J,v_J)$ as a tuple of $v_{\head(a)}\times v_{\tail(a)}$ matrices, one for each arrow $a$ in $Q_J^*$, then 
 \begin{equation}
 \label{eqn:varphiBa}
 \varphi(B)_{\beta_J(a)} = B_{(\iota_J\circ \beta_J)(a)}.
 \end{equation}
  
 Therefore $g=(g_i)\in \GL(v_J)\cong \GL(v)/G_K$ satisfies 
 \[
 \xaij{\beta_J(a)}{ij}(\varphi(g\cdot B)) = \psi(\xaij{\beta_J(a)}{ij})(g\cdot B) = \xpij{(\iota_J\circ \beta_J)(a)}{ij}(g\cdot B) = (g_{\head(a)}B_{(\iota_J\circ \beta_J)(a)}g_{\head(a)})_{ij}
 \]
 for each $1\leq i\leq v_{\head(a)}$ and $1\leq j\leq v_{\tail(a)}$; that is, 
 \[
 \varphi(g\cdot B)_{\beta_J(a)} = g_{\head(a)}B_{(\iota_J\circ \beta_J)(a)}g_{\tail(a)}^{-1}.
 \]
 
  On the other hand, let the same $g\in \GL(v_J)$ act on $\varphi(B)=(\varphi(B)_{\beta_J(a)})$ and apply \eqref{eqn:varphiBa} to get
 \[
  (g\cdot \varphi(B))_{\beta_J(a)} = g_{\head(a)} \varphi(B)_{\beta_J(a)} g_{\tail(a)}^{-1} = g_{\head(a)} B_{(\iota_J\circ\beta_J)(a)} g_{\tail(a)}^{-1}.
\]
Therefore $\varphi(g\cdot B)_{\beta_J(a)}  = (g\cdot \varphi(B))_{\beta_J(a)}$ for each arrow $a$ in $Q_J^*$, so $\varphi(g\cdot B)  = g\cdot \varphi(B)$ as required. 
 \end{proof}

  Recall from \eqref{eqn:Thetav} that $(\GL(v)/\kk^\times)^\vee\otimes_\ZZ \QQ$ is identified with the rational vector space $\Theta_v$ containing stability parameters. 
Similarly, we identify $\big(\GL(v_J)/\kk^\times\big)^\vee\otimes_\ZZ \QQ$ with the rational vector space
 \[
  \Theta_{v_J}:= \big\{\eta\in \Hom_{\ZZ}(\ZZ\oplus \ZZ^J,\QQ)\mid \eta(v_J)=0\big\}.
 \]
The quotient $\GL(v)\rightarrow \GL(v)/G_K\cong \GL(v_J)$ induces the inclusion $\Theta_{v_J}\hookrightarrow \Theta_v$ sending $\eta$ to $\overline{\eta}$, where
\[
\overline{\eta}_k = \left\{\begin{array}{cr} \eta_k & \text{if }k\in \{\infty\}\cup J \\ 0 & \text{if }k\in K,\end{array}\right.
\]
 that is, we simply augment $\eta$ with additional zeroes for each vertex of $Q$ that lies in the set $K$. 

Results of King~\cite{King94} imply that the GIT quotient $\Rep(A_J,v_J)\git_\eta \GL(v_J)$ is the coarse moduli space of $\eta$-semistable $A_J$-modules of dimension vector $v_J$ up to $\Seshadri$-equivalence. Since $v_J$ is primitive, $\Rep(A_J,v_J)\git_\eta \GL(v_J)$ is the fine moduli space of $\eta$-stable $A_J$-modules of dimension vector $v_J$ (up to isomorphism) whenever $\eta$ is \emph{generic}, that is, whenever every $\eta$-semistable $A_J$-module of dimension vector $v_J$ is $\eta$-stable. 
 
\begin{theorem}
\label{thm:closedimmersion}
  Let $J\subseteq \{0,1,\dots, r\}$ be non-empty. For $\eta\in \Theta_{v_J}$ satisfying $\eta_j\geq 0$ for all $j\in J\setminus \{0\}$ and $\eta_\infty\leq 0$, the $\GL(v_J)$-equivariant closed immersion $\varphi$ induces a commutative diagram 
  \begin{equation}
\label{eqn:commutativeVGIT}
\begin{tikzcd}
 \mathfrak{M}_{\overline{\eta}} \ar[d]\ar[r,"\varphi_\eta"] & \Rep(A_J,v_J)\git_\eta \GL(v_J) \ar[d] \\
 \Sym^n(\mathbb{A}^2/\Gamma) \ar[r,"\varphi_0"] &  \Rep(A_J,v_J)\git \GL(v_J), 
  \end{tikzcd}
  \end{equation}
   where the vertical maps are projective morphisms and the horizontal maps are closed immersions.
   \end{theorem} 

 \begin{proof}
   The conditions on $\eta$ are equivalent to having $\overline{\eta}\in F$, so Proposition~\ref{prop:AtoPi} gives an isomorphism $\mathfrak{M}_{\overline{\eta}}\cong \Rep(A,v)\git_{\overline{\eta}} \GL(v)$ of schemes over $\Sym^n(\mathbb{A}^2/\Gamma)$.  Define $K\subset Q_0^*$ as in \eqref{eqn:partition}. The map $\varphi$ from Lemma~\ref{lem:epi} is $\GL(v_J)$-equivariant, so each $\eta\in \Theta_{v_J}$ determines a closed immersion      
     \begin{equation}
         \label{eqn:GmodHisom}
   \big(\Rep(A,v)\git G_K\big)\git_{\eta} \big(\GL(v)/G_K\big) \longrightarrow \Rep(A_J,v_J)\git_\eta \GL(v_J) 
    \end{equation}
   of GIT quotients. Let $f\in \kk[\Rep(A,v)]$ and consider $k\geq 0$. Then $f$ is $k\overline{\eta}$-semi-invariant with respect to the $\GL(v)$-action on $\kk[\Rep(A,v)]$ if and only if $f$ is both $G_K$-invariant and, when regarded as an element of $\kk[\Rep(A,v)]^{G_K}$, $f$ is $k\eta$-semi-invariant with respect to the action of $\GL(v)/G_K$. Therefore, if we write subscripts to indicate taking the appropriate graded slice, then
   \begin{eqnarray}
   \big(\Rep(A,v)\git G_K\big)\git_{\eta} \big(\GL(v)/G_K\big) & = & \Proj \bigoplus_{k\geq 0} \big(\kk[\Rep(A,v)]^{G_K}\big)_{k\eta} \nonumber\\
   & \cong &   \Proj \bigoplus_{k\geq 0} \kk[\Rep(A,v)]_{k\overline{\eta}} \nonumber \\
   & = & \Rep(A,v)\git_{\overline{\eta}} \GL(v), \label{eqn:GGGisom}
   \end{eqnarray}
 as schemes over $\Rep(A,v)\git \GL(v)$ which, by Proposition~\ref{prop:M0RepAv}, is isomorphic to $\Sym^n(\mathbb{A}^2/\Gamma)$. The composition of the maps \eqref{eqn:GmodHisom} and \eqref{eqn:GGGisom} together with the isomorphism from the first sentence of this proof defines $\varphi_\eta$. Combining the special case $\eta=0$ with Proposition~\ref{prop:M0RepAv}  defines the lower horizontal map from \eqref{eqn:commutativeVGIT}, and we complete the diagram with projective maps given by VGIT. 
 \end{proof}

 \begin{remark}
 \label{rem:semistable}
The proof of Theorem~\ref{thm:closedimmersion} shows that the restriction of the categorical quotient of $\Rep(A,v)$ by $G_K$ defines a morphism $\Rep(A,v)^{\overline{\eta}}\to (\Rep(A,v)\git G_K)^\eta$ from the $\overline{\eta}$-semistable locus in $\Rep(A,v)$ to the $\eta$-semistable locus in $\Rep(A,v)\git G_K$.  
 \end{remark}
  
The set of stability parameters of the form $\overline{\eta}\in \Theta_v$ that arise in Theorem~\ref{thm:closedimmersion} with $J\neq \{0,\dots, r\}$ is almost equal to the boundary of $F$. The exceptions are: 
 \begin{enumerate}
 \item the only stability conditions in the boundary of $F$ that are not of this form are those in the relative interior of $F\cap \delta^\perp =\{\theta\in F \mid \theta(\delta)=0\}$. In that case, Kuznetsov~\cite{Kuznetsov07} observed that $
 \mathfrak{M}_\theta \cong \Sym^n(S)$ for $\theta\in \operatorname{relint}(F\cap \delta^\perp)$, where $S\to \mathbb{A}^2/\Gamma$ is the minimal resolution; and
 \item the only stability conditions of this form that do not lie in the boundary of $F$ are those in the relative interior of the facet of the closed cone $\overline{C_+}$ contained in 
 $\{\theta \mid \theta(\rho_0)=0\}$.
 \end{enumerate} 
 
 \begin{example}
     To illustrate this set of stability conditions in type $A_2$ with $n=3$, we draw a slice of the cone $F$ and its decomposition into GIT cones in Figure~\ref{fig:etabar}.  The lines in red indicate the intersection of $F$ with the hyperplanes $\{\theta \mid \theta(\rho_k)=0\}$ for $k=0,1,2$; only the hyperplane for $k=0$ passes through the interior of $F$ (and that's the only case where $0\not\in J$). 
   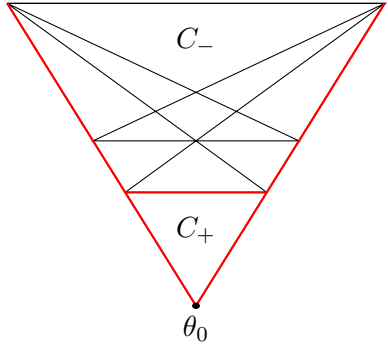
\begin{figure}[!ht]
   \centering
       \begin{tikzpicture}[baseline={(0,0)},xscale=1.25,yscale=1]
			\tikzset{>=latex}
            \draw (4,4)--(2,0)--(0,4)
            --(4,4)
            --(10/11,24/11)
            --(34/11,24/11)
            --(0,4)
            --(2.75,1.5)
            --(1.25,1.5)
            --(4,4)
            --(0,4)
            -- cycle;
            \draw [line width = 0.3mm, color=red] (4,4)--(2,0);
            \draw [line width = 0.3mm, color=red] (0,4)--(2,0);
            \draw [line width = 0.3mm, color=red] (2.75,1.5)--(1.25,1.5);
            \node at (2,1) {$C_+$};
            \node at (2,3.5) {$C_-$};
 \filldraw(2,0) circle (1pt) node[align=left,   below] {$\theta_0$};
			\end{tikzpicture}
           \caption{GIT cones in  $F$ for $n=3$ and type $A_2$, with parameters $\overline{\eta}$  drawn in red}
            \label{fig:etabar}
  \end{figure}

 \end{example}

 \section{Hilbert schemes and Quot schemes for Kleinian singularities}

 We now provide a geometric interpretation for some of the quiver moduli spaces associated to the algebra $A_J$ that appear in Theorem~\ref{thm:closedimmersion}, with a focus on the case $J=\{0\}$.
 
\subsection{Quiver construction of the Hilbert scheme of points}  \label{sec:relativeHilb} 
 For an integer $n\geq 1$, the Hilbert scheme of $n$-points on $\mathbb{A}^2/\Gamma$, denoted $\Hilb^n(\mathbb{A}^2/\Gamma)$, is the scheme representing the functor sending each $\kk$-scheme $T$ to the set 
 \[
 \left\{ Z \subset (\mathbb{A}^2/\Gamma)\times T \mid Z\to T \text{ is flat, finite and surjective of degree }n \right\}
 \]
 (see \cite{FGA4} or \cite[Tag 0B94]{stacks-project}).  From an algebraic viewpoint, $\Hilb^n(\mathbb{A}^2/\Gamma)$ is the fine moduli space of cyclic $\kk[x,y]^{\Gamma}$-modules that are isomorphic to $\kk^n$ as a $\kk$-vector space.  The universal family $\mathcal{Z}_\Gamma$ over $\Hilb^n(\mathbb{A}^2/\Gamma)$ is a locally-free sheaf of rank $n$, together with a surjective sheaf homomorphism 
 \[
 \mathcal{O}_{\Hilb^n(\mathbb{A}^2/\Gamma)} \longrightarrow \mathcal{Z}_\Gamma,
 \]
 where the fibre over any closed point of $\Hilb^n(\mathbb{A}^2/\Gamma)$ is the corresponding $\kk[x,y]^{\Gamma}$-module $M\cong \kk^n$ together with a surjective $\kk[x,y]^{\Gamma}$-module homomorphism $\kk[x,y]^{\Gamma}\to M$ that picks out the unique (up to scalar) generator as a  $\kk[x,y]^{\Gamma}$-module.

   We now recall the GIT construction of $\Hilb^n(\mathbb{A}^2/\Gamma)$ from \cite{CGGS21}.    For $J=\{0\}$, the quiver $Q_J^*$ and the ideal of relations $\ker(\beta_J)$ for the presentation $\beta_J\colon \kk Q_J^*\to A_J$ from Lemma~\ref{lem:AJA} are described explicitly in Example~\ref{exa:A_Jfor0}, as is the representation scheme $\Rep(A_J, v_J)$ for $v_J=(1,n)$. The group $\GL(v_J)=\GL(1)\times \GL(n)$ acts by conjugation on $\Rep(A_J, v_J)$, and we identify the character 
   \begin{equation}
       \label{eqn:eta}
   \eta:= (-n,1)\in \GL(v_J)^\vee
   \end{equation}
   of $\GL(v_J)$ with the determinant character of $\GL(n)$ under the inclusion  $\GL(n)^\vee\to \GL(v_J)^\vee$ induced by projection to the second factor $\GL(v_J)\to  \GL(n)$.

\begin{proposition}
\label{prop:ReptoHilb}
 For $n\geq 1$, $J=\{0\}$ and the character $\eta$ as above, there is an isomorphism 
 \begin{equation}
\label{eqn:relativeHilb}
 \omega\colon \Rep(A_J,v_J)\git_\eta \GL(v_J)\longrightarrow \Hilb^n(\mathbb{A}^2/\Gamma)
 \end{equation} 
 of schemes over the affine quotient $\Rep(A_J,v_J)\git \GL(v_J)\cong \Sym^n(\mathbb{A}^2/\Gamma)$.
 \end{proposition}
 \begin{proof}
 The isomorphism $\omega$ is inverse to the isomorphism constructed in \cite[Proposition~6.2]{CGGS21}. The induced isomorphism $\Rep(A_J,v_J)\git \GL(v_J)\to \Sym^n(\mathbb{A}^2/\Gamma)$ of affine varieties was first written down explicitly by Galluzzi--Vaccarino~\cite[Section~7]{GV10}.
  \end{proof}

  \subsection{Main results for Kleinian singularities}

 As a corollary, we can now prove Theorem~\ref{thm:Hilbnintro}.

    \begin{proof}[Proof of Theorem~\ref{thm:Hilbnintro}]
 For $J=\{0\}$ and $v_J=(1,n)$, the space of stability conditions $\Theta_{v_J}$ is the $\QQ$-span of the vector $\eta$ from \eqref{eqn:eta}, and the inclusion $\Theta_{v_J}\hookrightarrow \Theta_v$ sends $\eta$ to $\overline{\eta}=(-n,1,0,0,\dots, 0)$. Composing the closed immersion $\varphi_\eta$ from Theorem~\ref{thm:closedimmersion} with the isomorphism from Proposition~\ref{prop:ReptoHilb} determines a commutative diagram
   \begin{equation}
\label{eqn:quivertoHilb}
\begin{tikzcd}
 \mathfrak{M}_{\overline{\eta}} \ar[rr,hook,"f"]\ar[dr,"g"] & & \Hilb^{n}(\mathbb{A}^2/\Gamma) \ar[dl,swap, "h"] \\
 & \Sym^n(\mathbb{A}^2/\Gamma) &
  \end{tikzcd}
  \end{equation}
   where the diagonal maps are projective morphisms and $f$ is a closed immersion. 
   
  The non-degeneracy locus $U\subset \Sym^n(\mathbb{A}^2/\Gamma)$ parametrising $n$ distinct unordered points in $\mathbb{A}^2/\Gamma$ is open and dense. The morphism $g$ is surjective and the variety $\mathfrak{M}_{\overline{\eta}}$ is irreducible by Lemma~\ref{lem:reduced}, so $g^{-1}(U)$ is open and dense. The image of $f$ is therefore contained in the unique irreducible component of $\Hilb^n(\mathbb{A}^2/\Gamma)$ containing $h^{-1}(U)$, and since $f$ is a closed immersion, it identifies $\mathfrak{M}_{\overline{\eta}}$ with the closure of $h^{-1}(U)$, that is, with the good component of $\Hilb^{n}(\mathbb{A}^2/\Gamma)$. The fact that the good component is normal is immediate from Lemma~\ref{lem:reduced}.
  
  When the base field is $\mathbb{C}$, the fact that the good component of $\Hilb^{n}(\mathbb{A}^2/\Gamma)$ has symplectic singularities is also immediate from Lemma~\ref{lem:reduced}. For the symplectic resolution, the stability condition $\overline{\eta}$ lies in the closure of precisely one GIT chamber $C_+$ in the cone $F$ by \cite[Lemma~2.3]{BC20}, so the identification of $F$ with the movable cone of the Hilbert scheme of $n$-points on the minimal resolution of $\mathbb{A}^2/\Gamma$ from \cite[Theorem~1.2]{BC20} shows that $\mathfrak{M}_{\overline{\eta}}$ admits a unique projective symplectic resolution $\mathfrak{M}_\theta \cong n\Gamma\text{-Hilb}(\mathbb{A}^2)$ for $\theta\in C_+$, and hence so does the good component of $\Hilb^n(\mathbb{A}^2/\Gamma)$.   
   \end{proof}

  \begin{corollary}
  \label{cor:irreducible}
 For $n\geq 1$, the underlying reduced subscheme $\Hilb^{n}(\mathbb{A}^2/\Gamma)_{\operatorname{red}}$ is isomorphic to the normal variety $\mathfrak{M}_{\theta_0}$. When the base field is $\mathbb{C}$, it has symplectic singularities and admits a unique projective, symplectic resolution.
   \end{corollary}
 \begin{proof}
Zheng~\cite{Zheng23} proved that $\Hilb^n(\mathbb{A}^2/\Gamma)$ is irreducible\footnote{See Section~\ref{sec:CGGS21} for an alternative proof of irreducibility that is independent of the work of Zheng~\cite{Zheng23}.}, so the good component of  $\Hilb^{n}(\mathbb{A}^2/\Gamma)$ coincides with the underlying reduced subscheme. The result follows from Theorem~\ref{thm:Hilbnintro}.
\end{proof}

\begin{theorem}
 \label{thm:reduced}
 For $n\leq 7$, the scheme $\Hilb^n(\mathbb{A}^2/\Gamma)$ is actually reduced.
 \end{theorem}
 
 \begin{proof}
 Recent work of Hu~\cite[Theorem~1.6]{Hu25} shows that $\Hilb^n(\mathbb{A}^3)$ is Gorenstein for $n\le7$, and hence Cohen-Macaulay. The results now follows from Proposition~\ref{prop:CMreduced}.
 \end{proof}

\begin{proposition}
\label{prop:CMreduced}
 Let $n\geq 1$. If  $\Hilb^n(\mathbb{A}^3)$ is Cohen--Macaulay, then $\Hilb^n(\mathbb{A}^2/\Gamma)$ is reduced.
\end{proposition}
 \begin{proof}
      The coordinate ring of $\mathbb{A}^2/\Gamma$ is isomorphic to $\kk[z_1, z_2, z_3]/(f_\Gamma)$ by \eqref{eqn:ADEhypersurface}, so there is a closed immersion $\iota\colon \mathbb{A}^2/\Gamma\hookrightarrow \mathbb{A}^3$ as a hypersurface. Let $\mathcal{Z}_\Gamma\subset \Hilb^n(\mathbb{A}^2/\Gamma)\times \mathbb{A}^2/\Gamma$ be the universal subscheme for $\Hilb^n(\mathbb{A}^2/\Gamma)$. Then $(\id_{\Hilb^n(\mathbb{A}^2/\Gamma)}\times \iota)_*\mathcal{O}_{\mathcal{Z}_\Gamma}$ gives a flat family of subschemes of $\mathbb{A}^3$ of length $n$ over $\Hilb^n(\mathbb{A}^2/\Gamma)$, and the universal morphism 
 \begin{equation}
     \label{eqn:Hilbnclosedimmersion}
\Hilb^n(\mathbb{A}^2/\Gamma)\longrightarrow \Hilb^n(\mathbb{A}^3)
 \end{equation}
 is a closed immersion by \cite[Tag 0B97]{stacks-project}. We claim that this closed immersion realises  $\Hilb^n(\mathbb{A}^2/\Gamma)$ as a local complete intersection in $\Hilb^n(\mathbb{A}^3)$. Since $\Hilb^n(\mathbb{A}^3)$ is Cohen--Macaulay by assumption, the claim implies that $\Hilb^n(\mathbb{A}^2/\Gamma)$ is also Cohen--Macaulay and hence $S_1$. The result follows by Serre's criteria because  $\Hilb^n(\mathbb{A}^2/\Gamma)$ is generically smooth and hence $R_0$.
 
 To prove the claim, let $ \mathcal{Z}\subseteq \Hilb^n(\mathbb{A}^3)\times \mathbb{A}^3$ be the universal subscheme for $\Hilb^n(\mathbb{A}^3)$ and write $p\colon \mathcal{Z}\to \Hilb^n(\mathbb{A}^3)$ and $q\colon \mathcal{Z}\to \mathbb{A}^3$ for the first and second projections.  Since $p$ is finite and flat of degree $n$, the sheaf $p_* \mathcal{O}_{\mathcal{Z}}$ is locally-free of rank $n$ on $\Hilb^n(\mathbb{A}^3)$. The idea is to show that the image of \eqref{eqn:Hilbnclosedimmersion} is cut out locally by the section $p_*(q^*(f_\Gamma))$ of the locally-free sheaf $p_*\mathcal{O}_{\mathcal{Z}}$ of rank $n$ on $\Hilb^n(\mathbb{A}^3)$, so it's defined locally by $n$ equations. More formally, choose a cover by open affines $T=\Spec(R)\subseteq \Hilb^n(\mathbb{A}^3)$ on which $p_*\mathcal{O}_{\mathcal{Z}}$ can be trivialised via $R$-module isomorphisms
  \[
  \varphi_T\colon \Gamma\big(T,p_*\mathcal{O}_{\mathcal{Z}}\big)\longrightarrow R^{\oplus n}.
  \]
  Adopting the notation from the proof of \cite[Tag 0B97]{stacks-project}, define $Z'$ via the fibre square
   \begin{equation}
\label{eqn:lciHilb}
\begin{tikzcd}
   Z'\ar[d]\ar[r] & p^{-1}(T)\ar[d,"{(p,q)}"] \\
  T\times \mathbb{A}^2/\Gamma \ar[r,"{\id\times \iota}"] & T\times \mathbb{A}^3.
  \end{tikzcd}
  \end{equation}
  If $g\in B:= \Gamma(T,p_*\mathcal{O}_{\mathcal{Z}})$ denotes the restriction of $q^*(f_\Gamma)$ to $p^{-1}(T)$, then $Z'\cong \Spec(B/\langle g\rangle)$, and furthermore, the resulting $n$-tuple of functions $\varphi_T(g)=(g_1, \dots, g_n)$ cuts out the closed subscheme $\Spec(R/\langle g_1, \dots, g_n\rangle) \subseteq T$ which defines the intersection of $T$ with the image of the morphism \eqref{eqn:Hilbnclosedimmersion}. In particular, $\Hilb^n(\mathbb{A}^2/\Gamma)$ is locally defined by $n$ functions. The scheme $\Hilb^n(\mathbb{A}^2/\Gamma)$ is irreducible of dimension $2n$ (see Corollary~\ref{cor:irreducible} above), and for $n\le7$,  $\Hilb^n(\mathbb{A}^3)$ is irreducible of dimension $3n$ by Mazzola~\cite{Mazzola80}, so $\Hilb^n(\mathbb{A}^2/\Gamma)$ is a local complete intersection in $\Hilb^n(\mathbb{A}^3)$ as claimed. 
 \end{proof}

 \subsection{A universal morphism}
 \label{sec:CGGS21}
 We now describe a situation where the map $\varphi_\eta$ from Theorem~\ref{thm:closedimmersion} is universal. This section is not required for the proofs of Theorems~\ref{thm:HilbnSintro}, \ref{thm:nleq7intro}, \ref{thm:Hilbnintro} or Corollary~\ref{cor:Beauville}.
 
 In this section only, let $\eta\in \Theta_{v_J}$ satisfy $\eta_j>0$ for \emph{all} $j\in J$, hence $\eta_\infty<0$. Every such stability condition $\eta$ is generic, and since $v_J$ is indivisible, it follows from King~\cite{King94} that $\Rep(A_J,v_J)\git_\eta \GL(v_J)$ is the fine moduli space of $\eta$-stable $A_J$-modules of dimension vector $v_J$ up to isomorphism. The corresponding stability condition $\overline{\eta}\in \Theta_v$ lies in the closure of the GIT chamber
 \[
 C_+:=\{\theta\in \Theta_v \mid \theta_i>0 \text{ for all }0\leq i\leq r\}
 \]
 that is contained in the cone $F$ from \eqref{eqn:F}, where $\mathfrak{M}_{\theta}\cong n\Gamma\text{-}\Hilb(\mathbb{A}^2)$
  for $\theta\in C_+$. In particular, variation of GIT quotient defines a projective, symplectic resolution of singularities $\pi_J\colon \mathfrak{M}_\theta\rightarrow \mathfrak{M}_{\overline{\eta}}$.
  
  \begin{theorem}
  \label{thm:CGGS21}
   For non-empty $J\subseteq \{0,1,\dots, r\}$, assume that $\eta\in \Theta_{v_J}$ satisfies $\eta_j>0$ for all $j\in J$. Then for any $\theta\in C_+$, there is a commutative diagram    \begin{equation}
\label{eqn:CGGS1}
\begin{tikzcd}
\mathfrak{M}_\theta\ar[d,swap,"\pi_J"]\ar[dr,"\tau_J"] & \\
 \mathfrak{M}_{\overline{\eta}} \ar[r,"\varphi_\eta"] & \Rep(A_J,v_J)\git_\eta \GL(v_J) 
  \end{tikzcd}
  \end{equation}
   where $\tau_J$ is a universal morphism to the fine moduli space and $\varphi_\eta$ induces an isomorphism of the underlying reduced schemes. In particular, $\Rep(A_J,v_J)\git_\eta \GL(v_J)$ is irreducible and the underlying reduced scheme is normal. In addition, over the ground field $\mathbb{C}$, it has symplectic singularities.
  \end{theorem}
  \begin{proof}
Let $K$ satisfy \eqref{eqn:partition}. Combining the quotient map for the $G_K$-action on $\Rep(A,v)$ with the $\GL(v_J)$-equivariant morphism from Lemma~\ref{lem:epi} gives a commutative diagram of affine schemes
\begin{equation}
\label{eqn:affinecommutative}
\begin{tikzcd}
 \Rep(A,v) \ar[d,swap,"\pi"]\ar[dr,"\tau"] & \\
 \Rep(A,v)\git G_K \ar[r,"\varphi"] & \Rep(A_J,v_J) 
  \end{tikzcd}
  \end{equation}
  for $\tau:=\varphi\circ \pi$. The map $\pi$ is equivariant with respect to actions of $\GL(v)$ and $\GL(v_J)\cong \GL(v)/G_K$ on $\Rep(A,v)$ and $\Rep(A,v)\git G_K$ respectively, so $\tau$ is also equivariant with respect to the actions of $\GL(v)$ and $\GL(v_J)$. The open immersion  of semistable loci $\Rep(A,v)^{\theta}\to \Rep(A,v)^{\overline{\eta}}$ induces a VGIT morphism because  $\overline{\eta}$ lies in the closure of the GIT chamber $C_+$ containing $\theta$, and this gives the VGIT morphism $\pi_J$ by Proposition~\ref{prop:AtoPi} because $\theta, \overline{\eta}\in F$. Furthermore, the morphism $\Rep(A,v)^{\overline{\eta}}\to (\Rep(A,v)\git G_K)^\eta$ from Remark~\ref{rem:semistable} obtained by restricting $\varphi$ to the $\overline{\eta}$-semistable locus induces the morphism $\varphi_\eta$ after passing to the quotients by $\GL(v_J)$. Thus, we obtain  diagram \eqref{eqn:CGGS1} from \eqref{eqn:affinecommutative} by setting $\tau_J:= \varphi_\eta\circ \pi_J$.

Now let $T_J:= \bigoplus_{i\in \{\infty\}\cup J\}} T_i$ denote the tautological locally-free sheaf on the fine moduli space $\Rep(A_J,v_J)\git_\eta ~\GL(v_J)$, where $T_\infty$ is the trivial line bundle and $T_j$ has rank $n\dim(\rho_j)$ for $j\in J$. To see that $\tau_J$ is universal, it suffices by \cite[Lemma~4.1]{CGGS21} to show that 
\[
\tau_J^*(T_i)\cong \mathcal{R}_i\quad \text{for all }i\in \{\infty\}\cup J. 
\]
Clearly $\tau_J^*(T_\infty)\cong \mathcal{O}_{\mathfrak{M}_\theta}\cong \mathcal{R}_\infty$. For $j\in J$, the bundle
$T_j$ on $\Rep(A_J,v_J)\git_\eta \GL(v_J)$ is obtained by descent from the $\GL(v_J)$-equivariant vector bundle $\varrho_j\vert_{\GL(v_J)}\otimes \mathcal{O}_{\Rep(A_J,v_J)^{\eta}}$ on the $\eta$-stable locus in $\Rep(A_J,v_J)$. Since $\varphi$ is a $\GL(v_J)$-equivariant closed immersion, the bundle
\begin{equation}
\label{eqn:varrhotensor} \varphi^*\big(\varrho_j\vert_{\GL(v_J)}\otimes \mathcal{O}_{\Rep(A_J,v_J)^{\eta}}\big)
=
\varrho_j\vert_{\GL(v_J)}\otimes \mathcal{O}_{(\Rep(A,v)\git G_K)^{\eta}}
\end{equation}
 descends to $\varphi_\eta^*(T_j)$, the restriction of $T_j$ to the closed subvariety $\mathfrak{M}_{\overline{\eta}}$ of $\Rep(A_J,v_J)\git\eta \GL(v_J)$. The $\GL(v)$-equivariant sheaf $\varrho_j\otimes \mathcal{O}_{\Rep(A,v)}$ descends via $\pi$ to $\varrho_J\vert_{\GL(v_J)}\otimes \mathcal{O}_{\Rep(A,v)\git G_K}$ because $G_K$ acts trivially on each fibre. Therefore, pulling back the sheaf \eqref{eqn:varrhotensor} along the restriction of $\pi$ to the $\overline{\eta}$-semistable locus following Remark~\ref{rem:semistable} shows that 
 \begin{equation}
\label{eqn:varrhoagain}  (\varphi\circ \pi)^*\big(\varrho_j\vert_{\GL(v_J)}\otimes \mathcal{O}_{\Rep(A_J,v_J)^{\eta}}\big)
=
\varrho_j\otimes \mathcal{O}_{\Rep(A,v)^{\overline{\eta}}}
\end{equation}
 Finally, we need only pullback along the open immersion $\Rep(A,v)^{\theta}\to \Rep(A,v)^{\overline{\eta}}$ to see that the bundle $(\varphi_\eta\circ \pi_J)^*(T_j)$ on $\mathfrak{M}_\theta$ is obtained by descent from the $\GL(v)$-equivariant bundle $\varrho_j\otimes \mathcal{O}_{\Rep(A,v)^\theta}$ on the $\theta$-stable locus in $\Rep(A,v)$. Therefore $(\varphi_\eta\circ \pi_J)^*(T_j)\cong \mathcal{R}_j$ as required.

 Since $\mathfrak{M}_{\overline{\eta}}$ is irreducible, normal and has symplectic singularities over $\mathbb{C}$, it  remains to show that the closed immersion $\varphi_\eta$ is surjective on closed points. The proof of this assertion runs exactly as in the final paragraph in the proof of \cite[Theorem~4.8]{CGGS21}.
\end{proof}

\begin{remark}
 The statement of Theorem~\ref{thm:CGGS21} has appeared before:
 \begin{enumerate}
     \item It appeared in work of Craw, Gyenge, Gammelgaard and Szendr\H{o}i~\cite[Theorem~4.8]{CGGS21}, but Yehao Zhou kindly pointed out a gap in the proof of \cite[Lemma~4.3]{CGGS21} on which the original proof of \cite[Theorem~4.8]{CGGS21} relied. For $J=\{0\}$ and $\eta=(-n,1)$, Propositions~\ref{prop:ReptoHilb} and Theorem~\ref{thm:CGGS21}  show that $\Hilb^n(\mathbb{A}^2/\Gamma)$ is irreducible. This provides a complete, self-contained proof of the main result \cite[Theorem~1.1]{CGGS21} that is independent of the work by  Zheng~\cite{Zheng23}.
 \item For any non-empty $J\subseteq \{0,1,\dots, r\}$, \cite[Proposition~4.3]{CGGS2} shows that the fine moduli space $\Rep(A_J,v_J)\git_\eta \GL(v_J)$ from Theorem~\ref{thm:CGGS21} is isomorphic to the \emph{orbifold Quot scheme} for the stack $[\mathbb{A}^2/\Gamma]$, denoted $\Quot^{v_J}_J\big([\mathbb{A}^2/\Gamma]\big)$. The same gap appears in \cite[Lemma~5.5]{CGGS2}, but again, Theorem~\ref{thm:CGGS21} bypasses this to show that $\Quot^{v_J}_J\big([\mathbb{A}^2/\Gamma]\big)$ is irreducible, and its underlying reduced subscheme is normal. Similar  results hold for orbifold Quot schemes associated to dimension vectors other than $v_J$, see \cite{BG25,Craw25}.
 \end{enumerate}
  \end{remark}

      \section{On quasi-projective surfaces with canonical singularities}
        \label{sec:6}

 We conclude by studying the \'{e}tale-local structure of Hilbert schemes of points on affine schemes of finite type over $\kk$ in order to prove Theorems~\ref{thm:HilbnSintro}-\ref{thm:nleq7intro} and Corollary~\ref{cor:Beauville}. In passing, we present a self-contained proof of the result by Fogarty~\cite{Fogarty68} that $\Hilb^n(S)$ is nonsingular for a non-singular quasi-projective surface $S$.
 
\subsection{Formal neighbourhoods of points on punctual Hilbert schemes}
Let $A$ be a finitely generated commutative $\kk$-algebra. For $S:=\Spec A$ and $n\geq 1$, we write $S^{[n]}:=\Hilb^{n}(S)$ for the Hilbert scheme of $n$ points on $S$. Let $U\subset S\times S^{[n]}$ denote the universal subscheme. For each closed point $q\in S^{[n]}$, the fibre of $q$ under the natural morphism $U\to S^{[n]}$ is a \emph{cluster} in $S$, that is, a 0-dimensional closed subscheme $Z\subset S$ of length $n$. For each cluster $Z\subset S$, we write $[Z]\in S^{[n]}$ for the corresponding closed point. 

For any closed point $p\in S$ corresponding to a maximal ideal $\mathfrak{m}\subset A$, we denote the affine scheme $\Spec\widehat{A}_\mathfrak{m}$ by $\widehat{S}_p$, where $\widehat{A}_\mathfrak{m}$ is the completion of $A$ at $\mathfrak{m}$.

\begin{lemma}
\label{lem:localtoglobal}
 For $[Z]\in S^{[n]}$, write $Z =\bigsqcup_{1\leq k\leq \ell} Z_k$ where $p_k:=\supp(Z_k)$ and $m_k:= \dim H^0(\mathcal{O}_{Z_k})$. There is a morphism of schemes
 \[
  g\colon\prod_{1\leq k\leq \ell} (\widehat{S}_{p_k})^{[m_k]}\longrightarrow  S^{[n]}.
  \]
\end{lemma}
\begin{proof}
The union of the universal subschemes for the flat families over $(\widehat{S}_{p_k})^{[m_k]}$ for $1\leq k\leq \ell$ is a flat family of clusters in $S$ of length $n$, so the universal property of $S^{[n]}$ defines the morphism.
\end{proof}

 Let $S^{(n)}:=\Sym^n(S)$ denote the $n$-th symmetric product of $S$. The \emph{Hilbert-Chow morphism} \[
\pi\colon S^{[n]}\longrightarrow S^{(n)}
\]
 maps each cluster $Z$ in $S$ to the corresponding cycle $\pi(Z)=\sum_{p\in \supp(Z)} m_p p$, where $\supp(Z)$ is the support of $Z$ and where the multiplicities $m_p=\dim H^0(\mathcal{O}_{Z,p})$ satisfy $\sum_{p\in \supp(Z)} m_p=n$. Observe that the morphism from Lemma~\ref{lem:localtoglobal} depends only on the cycle $\pi(Z)$ associated to $Z$.

 Once and for all, let $\sigma:=\sum_{1\leq k\leq \ell} m_kp_k\in S^{(n)}$ be a cycle for which $p_1, \dots, p_\ell$ are distinct, and let $X_\sigma=\pi^{-1}(\sigma)$ in $S^{[n]}$ denote the scheme-theoretic fibre of the Hilbert--Chow morphism; this is the \emph{local punctual Hilbert scheme} \cite[Section~2.3.1]{Bertin08}. Similarly, for each $1\leq k\leq \ell$, write 
\[
\pi_k\colon S^{[m_k]}\longrightarrow  S^{(m_k)}
\]
 for the Hilbert--Chow morphism and define $X_{k}:=(\pi_{k})^{-1}(m_kp_k)$ to be the scheme-theoretic fibre. 
 
\begin{lemma}
\label{lem:localpunctual}
 There is an isomorphism $X_\sigma\cong \prod_{1\leq k\leq \ell} X_k$.
\end{lemma}
\begin{proof}
  The restriction of the universal subscheme $U\subset S\times S^{[n]}$ to $S\times X_\sigma$ is a flat family $U_\sigma$ of clusters parametrised by $X_\sigma$, each with support equal to $\{p_1,\dots,p_\ell\}\subset S$. This subscheme decomposes as a disjoint union $U_\sigma=\bigsqcup_{1\leq k\leq \ell} U_k$, where $U_k\subset S\times X_k$ is the flat family of clusters of length $m_k$ supported at $p_k$. The universal properties of the local punctual Hilbert schemes $X_k$ determine a morphism $X_\sigma\to \prod_{1\leq k\leq \ell} X_k$. Similarly, the universal property of $X_\sigma$ applied to the flat family $\bigsqcup_{1\leq k\leq \ell} U_k$ determines the inverse morphism $\prod_{1\leq k\leq \ell} X_k\to X_\sigma$. 
\end{proof}

 Let $\Spec(B)\subseteq S^{[n]}$ be an affine open subscheme that intersects $X_\sigma$, and write $U_B\subset S\times \Spec(B)$ for the restriction of the universal subscheme. Let $I\subset B$ denote the ideal satisfying $X_\sigma\cap \Spec(B)\cong \Spec(B/I)$, and let $U_{B/I}\subset S\times \Spec(B/I)$ denote the restriction of the universal subscheme. The affine scheme $\Spec(\widehat{B}_I)$ defined by the $I$-adic completion of $B$ fits into a commutative diagram
 \begin{equation}
 \label{eqn:univsubschemes}
 \begin{tikzcd}
U_{B/I}\ar[r,hook]\ar[d]    & U_{\widehat{B}_I} \ar[d]\ar[r]              &    U_B \ar[d] \\
\Spec(B/I)\ar[r,hook] & \Spec(\widehat{B}_I) \ar[r] & \Spec(B),
  \end{tikzcd}
  \end{equation}
 where $U_{\widehat{B}_{I}} = \Spec(\widehat{B}_I) \times_{\Spec(B)}U_B$ is defined so that both squares are cartesian.
The central vertical map is flat and finite of degree $n$, so $U_{\widehat{B}_{I}}$ is a flat family of clusters of length $n$ over $\Spec(\widehat{B}_I)$.

\begin{lemma}
\label{lem:f_B}
  The flat family $U_{\widehat{B}_I}$ determines a universal morphism 
  \[
  f_B\colon \Spec(\widehat{B}_I)\longrightarrow \prod_{1\leq k\leq \ell} (\widehat{S}_p)^{[m_k]}.
  \]
\end{lemma}
\begin{proof}
 Since the scheme-theoretic image of the natural morphism $U_{B/I}\to S$ is a 0-dimensional closed subscheme with support $\{p_1,\dots,p_\ell\}\subset S$, it is of the form $\Spec (A/J)$ for some ideal $J\subset A$ whose radical is the product of the maximal ideals $\mathfrak{m}_k$ for the points $p_k$  for $1\leq k\leq \ell$. Thus the morphism $U_B\to S$ induces a morphism $U_{\widehat{B}_I}\to \Spec (\widehat{A}_J)$. Since $\Spec (\widehat{A}_J)$ is naturally isomorphic to the scheme $\Spec \big(\prod_{1\leq k\leq\ell} \widehat{A}_{\mathfrak{m}_k}\big)=\bigsqcup_{1\leq k\leq \ell}\widehat{S}_{p_k}$, the universal subscheme $U_{\widehat{B}_I}$ over $\Spec(\widehat{B}_I)$ is a closed subscheme of $\big(\bigsqcup_{1\leq k\leq \ell}\widehat{S}_{p_k}\big)\times\Spec(\widehat{B}_I)$. This implies that  $U_{\widehat{B}_I}$ is the disjoint union of the universal subschemes for the Hilbert schemes $(\widehat{S}_p)^{[m_k]}$ for $1\leq k\leq \ell$ and, therefore, the universal properties of the fine moduli spaces $(\widehat{S}_p)^{[m_k]}$ define the desired morphism $f_B$.
 \end{proof}
 
  Now, consider the formal scheme $\widehat{S^{[n]}}_{X_\sigma}$ obtained as the formal completion of $S^{[n]}$ along $X_\sigma$, that is, the topological space $X_\sigma$ together with the sheaf of rings  $\mathcal{O}_{\widehat{S^{[n]}}} := \varprojlim \mathcal{O}_{S^{[n]}}/\mathscr{I}^n$, where $\mathscr{I}$ is the sheaf of ideals defining $X_{\sigma}$ in $S^{[n]}$.

\begin{proposition}
\label{prop:formalnbd}
 For $\sigma:=\sum_{1\leq k\leq \ell} m_kp_k\in S^{(n)}$,  the formal scheme  $\widehat{S^{[n]}}_{X_{\sigma}}$ is isomorphic to the formal completion of $\prod_{1\leq k\leq \ell} (\widehat{S}_{p_k})^{[m_k]}$ along $\prod_{1\leq k\leq \ell} X_k$.
In particular, the completion of $S^{[n]}$ along $X_{\sigma}$ depends only on the affine schemes $\widehat{S}_{p_1}, \dots \widehat{S}_{p_\ell}$ and the tuple $(m_1, \dots, m_\ell)\in \mathbb{N}^\ell$. 
\end{proposition}
\begin{proof}
 For $h\in B$, the isomorphisms $(B/I^m)_h\to (B_h/I^mB_h)$ for all $m\in \mathbb{N}$ induce a map of rings $(\widehat{B}_I)_h\to \widehat{(B_h)}_{IB_h}$ and hence a morphism $\Spec \widehat{(B_h)}_{IB_h}\to \Spec(\widehat{B}_I)_h$. The composition of this map with the restriction of $f_B$ to $\Spec((\widehat{B}_I)_h)$ coincides with the map $f_{B_h}$ from Lemma~\ref{lem:f_B}.
 Thus, the morphisms from Lemma~\ref{lem:f_B} glue to give a morphism
  \[
  f\colon \widehat{S^{[n]}}_{X_\sigma}\longrightarrow \prod_{1\leq k\leq \ell} (\widehat{S}_p)^{[m_k]}
  \]
  of formal schemes. Compare this with the morphism
  \[
  g\colon\prod_{1\leq k\leq \ell} (\widehat{S}_p)^{[m_k]}\longrightarrow  S^{[n]}
  \]
  from Lemma~\ref{lem:localtoglobal}. Under the identification from Lemma~\ref{lem:localpunctual}, completing the morphisms $f$ and $g$ along $X_\sigma$ and $\prod_{1\leq k\leq \ell} X_k$ respectively gives mutually inverse isomorphisms of formal schemes.
\end{proof}

\begin{corollary}
\label{cor:completelocal}
For $[Z]\in S^{[n]}$, write $Z =\bigsqcup_{1\leq k\leq \ell} Z_k$ where $\supp(Z_k)=p_k$ and $\operatorname{length}(\mathcal{O}_{Z_k})=m_k$. 
The following formal schemes are isomorphic:
\begin{enumerate}
\item[\one] the formal completion of $S^{[n]}$ at the closed point $[Z]$;
\item[\two] the formal completion of $\prod_{1\leq k\leq \ell} S^{[m_k]}$ at the closed point $\big([Z_1], \dots, [Z_\ell]\big)$; and 
\item[\three] the formal completion of $\prod_{1\leq k\leq \ell} (\widehat{S}_{p_k})^{[m_k]}$ at the closed point $\big([Z_1],\dots,[Z_\ell]\big)$.
\end{enumerate}
Moreover, the same statement holds after replacing $S^{[n]}, S^{[m_k]}$ and $(\widehat{S}_{p_k})^{[m_k]}$ for $1\leq k\leq \ell$ by their underlying reduced schemes respectively.
\end{corollary}
\begin{proof}
The isomorphism $X_\sigma\cong \prod_{1\leq k\leq \ell} X_k$ identifies $[Z]$ with $([Z_1], \dots, [Z_\ell])$, so the isomorphism between \one\ and \three\ follows from Proposition~\ref{prop:formalnbd}\two. The formal schemes in \two\ and \three\ both have the single point $([Z_1,\dots, [Z_\ell])$ as their underlying topological space, so it remains to establish an isomorphism of complete local rings
\begin{equation}
   \label{eqn:completeproduct}
\widehat{\mathcal{O}}_{\prod_k S^{[m_k]},([Z_1], \dots, [Z_\ell])} \cong 
\widehat{\mathcal{O}}_{\prod_k (\widehat{S}_{p_k})^{[m_k]},([Z_1], \dots, [Z_\ell])}.
\end{equation}
Applying  Proposition~\ref{prop:formalnbd}\two\ to each cycle $m_k p_k\in S^{(m_k)}$ gives an isomorphism of between  $\widehat{S^{[m_k]}}_{X_k}$ and the formal completion of $\widehat{S}_{p_k}^{[m_k]}$ along $X_k$. Taking the stalk of the structure sheaves under this isomorphism gives an isomorphism of complete local rings
\begin{equation}
    \label{eqn:completerings}
\widehat{\mathcal{O}}_{S^{[m_k]},[Z_k]} \cong 
\widehat{\mathcal{O}}_{(\widehat{S}_{p_k})^{[m_k]},[Z_k]}
\end{equation}
 for $1\leq k\leq \ell$. The local rings for schemes of finite type over $\kk$ are excellent \cite[Tag 07QU]{stacks-project}, so the result follows by applying the isomorphism from Lemma~\ref{lem:completeproduct} below to both complete local rings from \eqref{eqn:completeproduct}, and then applying the isomorphisms \eqref{eqn:completerings} to each factor.
 \end{proof}

 \begin{lemma}
 \label{lem:completeproduct} 
 For $1\leq k\leq \ell$, let $(A_k,\mathfrak{m}_k)$ be a Noetherian local $\kk$-algebra, and let $\widehat{A_k}$ be the $\mathfrak{m}_k$-adic completion of $A_k$. For $A:=A_1\otimes_{\kk}\cdots \otimes_{\kk}A_\ell$, 
define the ideal $I:=\mathfrak{m}_1A+\cdots+\mathfrak{m}_\ell A$
in $A$, and for $B:= \widehat{A_1}\otimes_{\kk}\cdots\otimes_{\kk} \widehat{A_\ell}$, 
define the ideal $J:=\mathfrak{m}_1B+\cdots+\mathfrak{m}_\ell B$ 
in $B$. The $I$-adic completion of $A$ is isomorphic to the $J$-adic completion of $B$. If in addition each $A_k$ is excellent with nilradical $N_k$, then the result holds if we replace each $A_k$ throughout by $A_k/N_k$.
 \end{lemma}
 \begin{proof}
  For $n\geq 1$, the ideals $I_n:= \mathfrak{m}_1^n A+\dots +\mathfrak{m}_\ell^n A$ satisfy $A/I_n\cong (A_1/\mathfrak{m}_1^n)\otimes_\kk \cdots \otimes_{\kk} (A_\ell/\mathfrak{m}_\ell^n)$. The $I$-adic topology on $A$ is equivalent to that defined by the system of ideals $\{I_n\}_{n\geq 1}$, so the $I$-adic completion of $A$ can be computed as the inverse limit 
  \begin{equation}
      \label{eqn:Ahat}
  \widehat{A}
    \cong\varprojlim{} (A_1/\mathfrak{m}_1^n)\otimes_\kk \cdots \otimes_{\kk} (A_\ell/\mathfrak{m}_\ell^n).
   \end{equation}
  Similarly, the ideals $J_n:= \mathfrak{m}_1^n B+\dots +\mathfrak{m}_\ell^n B$ satisfy $B/J_n\cong (\widehat{A_1}/\mathfrak{m}_1^n\widehat{A_1})\otimes_\kk \cdots \otimes_{\kk} (\widehat{A_\ell}/\mathfrak{m}_\ell^n\widehat{A_\ell})$, and the $J$-adic completion of $B$ can be computed as the inverse limit
   \begin{equation}
 \label{eqn:Bhat}
 \widehat{B}
 \cong\varprojlim{} (\widehat{A_1}/\mathfrak{m}_1^n\widehat{A_1})\otimes_\kk \cdots \otimes_{\kk} (\widehat{A_\ell}/\mathfrak{m}_\ell^n\widehat{A_\ell}).
  \end{equation}
 Since $A_k$ is Noetherian, exactness of completion for finitely generated $A_k$-modules ensures that the short exact sequence 
\[ 0\longrightarrow \mathfrak{m}_k^n\longrightarrow A_k\longrightarrow A_k/\mathfrak{m}_k^n\longrightarrow 0 \]
 determines an isomorphism $\widehat{A_k}/\mathfrak{m}_k^n\widehat{A_k}\cong A_k/\mathfrak{m}_k^n$ for all $1\leq k\leq \ell$. Now \eqref{eqn:Ahat} and \eqref{eqn:Bhat} give $\widehat{A}\cong \widehat{B}$.
 
 For the second claim, it suffices to show that taking the reduced ring structure commutes with both tensor product and completion.
 Since $\kk$ is algebraically closed and hence perfect, the tensor product of reduced rings is reduced \cite[Tags 034N, 030V]{stacks-project}, so the universal property of reducedness implies that taking the reduced ring structure commutes with tensor product. For commutativity with completion, let $R$ be a Noetherian local $\kk$-algebra with nilradical $N$. Exactness of completion for finitely generated $R$-modules gives 
 \begin{equation}
     \label{eqn:completionexact}
     \widehat{R/N}\cong \widehat{R}/N \widehat{R}.
      \end{equation}
 The quotient $R/N$ is reduced and excellent, so the map $R/N\to \widehat{R/N}$ is regular and hence $\widehat{R/N}$ is reduced by \cite[Tag 07QK]{stacks-project}. The nilradical of $\widehat{R}$ is therefore contained in $N\widehat{R}$ by  \eqref{eqn:completionexact}. The opposite inclusion is clear, so $N\widehat{R}$ is the nilradical of $\widehat{R}$. Therefore \eqref{eqn:completionexact} expresses precisely commutativity of completion with taking the reduced ring structure.
  \end{proof}
 
\subsection{Generalising Fogarty's theorem to canonical surfaces}
 Fix $\ell\in \mathbb{N}$. For $1\leq k\leq \ell$, let $\Gamma_k\subset \SL(2,\kk)$ be a (possibly trivial) finite subgroup. Let $0\in \mathbb{A}^2/\Gamma_k$ denote the origin when $\Gamma_k$ is trivial, and the unique singular point when $\Gamma_k$ is nontrivial. We also write \[
 \widehat{(\mathbb{A}^2/\Gamma_k)}_{0}
 \]
 for the spectrum of the completion of the local ring $\mathcal{O}_{\mathbb{A}^2/\Gamma_k,0}$. 

 In what follows, for any scheme $X$, we write $X_\mathrm{red}$ for the underlying reduced subscheme.
  
\begin{lemma}
\label{lem:completionKleinian}
 For $1\leq k\leq \ell$,  let $Z_k$ be a subscheme of length $m_k\in \mathbb{N}$ in the affine scheme $\widehat{(\mathbb{A}^2/\Gamma_k)}_{0}$. Then the completion of the local ring of 
 \[
 \prod_{1\leq k\leq \ell} \Hilb^{m_k}\big(\widehat{(\mathbb{A}^2/\Gamma_k)}_{0}\big)_\mathrm{red}
 \]
 at $([Z_1],\dots,[Z_\ell]\big)$ is normal, and it's regular if the subgroup $\Gamma_k$ of $\SL(2,\kk)$ is trivial for all $1\leq k\leq \ell$. 
 \end{lemma}

\begin{proof}
 For $1\leq k\leq \ell$, it is well known that the Hilbert scheme of $m_k$-points in the affine plane is nonsingular (see, for example, \cite{Craw21, NakajimaHilbert}). Thus, if each subgroup $\Gamma_k$ is trivial, then $\Hilb^{m_k}(\mathbb{A}^2/\Gamma_k)$ is nonsingular. Otherwise, some $\Gamma_k$ is nontrivial, and Corollary~\ref{cor:irreducible}  tells us that $\Hilb^{m_k}(\mathbb{A}^2/\Gamma_k)_\mathrm{red}$ is normal, so $\prod_{1\leq k\leq \ell} \Hilb^{m_k}(\mathbb{A}^2/\Gamma_k))_\mathrm{red}$ is normal. Zariski~\cite[Th\'{e}or\`{e}me~2]{Zariski50} implies that the complete local ring $\widehat{\mathcal{O}}_{X,p}$ of $X:= \prod_{1\leq k\leq \ell} \Hilb^{m_k}(\mathbb{A}^2/\Gamma_k))$ at the point $p:=([Z_1],...,[Z_\ell])$ is normal.  It follows from Corollary~\ref{cor:completelocal} that the  complete local ring $\widehat{\mathcal{O}}_{X',p}$ of 
 \[
 X':=\prod_{1\leq k\leq \ell} \Hilb^{m_k}\big(\widehat{(\mathbb{A}^2/\Gamma_k)}_0\big)_\mathrm{red}
 \]
 at the point $p=([Z_1],...,[Z_\ell])$ is normal.
 \end{proof}

\begin{proof}[Proof of Theorem~\ref{thm:HilbnSintro}]
Since $S$ is connected and quasi-projective, it follows that $S^{[n]}:=\Hilb^{n}(S)$ is also connected and quasi-projective (see, for example,  \cite[Lemma~3.7]{Lehn04}).

To see that $S^{[n]}_\mathrm{red}$ is normal, we may reduce to the case where $S$ is affine by \cite[Lemma~2.6]{Bertin08}. We need only show that the local ring $\mathcal{O}_{S^{[n]}_\mathrm{red},[Z]}$ is normal for each $[Z]\in S^{[n]}_\mathrm{red}$, and for this, it suffices by \cite[Tag 0FIZ]{stacks-project} to show that the complete local ring $\widehat{\mathcal{O}}_{S^{[n]}_\mathrm{red},[Z]}$ is normal. 
Corollary~\ref{cor:completelocal} shows that the completion of $S^{[n]}_\mathrm{red}$ at $[Z]$ is isomorphic to the completion of the product
\[
\prod_{1\leq k\leq \ell} \Hilb^{m_k}(\widehat{S}_{p_k})_\mathrm{red}
\]
at a specific point $([Z_1],\dots,[Z_\ell]\big)$, where $\widehat{S}_{p_k}$ is the spectrum of the completion of the local ring $\mathcal{O}_{S,p_k}$. Since $S$ has canonical singularities, the complete local ring $\widehat{\mathcal{O}}_{S,p_k}$ at $p_k\in S$ is isomorphic to the complete local ring at either the origin in the affine plane, or at 
the unique singular point of a Kleinian singularity $\mathbb{A}^2/\Gamma_k$ for some nontrivial finite subgroup $\Gamma_k\subset \SL(2,\kk)$. As above, we allow $\Gamma_k$ to be trivial, and write $0\in \mathbb{A}^2/\Gamma_k$ for the origin when $\Gamma_k$ is trivial, and the singular point when $\Gamma_k$ is nontrivial. Thus, each $\widehat{S}_{p_k}$ is isomorphic to the scheme $\widehat{(\mathbb{A}^2/\Gamma_k)}_{0}$ for some finite (possibly trivial) subgroup $\Gamma_k\subset \SL(2,\kk)$, so we must show that the complete local ring of 
\[
\prod_{1\leq k\leq \ell} \Hilb^{m_k}\big(\widehat{(\mathbb{A}^2/\Gamma_k)}_{0}\big)_\mathrm{red}
\]
at the point $([Z_1],\dots,[Z_\ell]\big)$ is normal.
 This was established in Lemma~\ref{lem:completionKleinian}, so $S^{[n]}_\mathrm{red}$ is normal.

Since the Noetherian scheme $S^{[n]}_\mathrm{red}$ is both connected and normal, \cite[Tag033M]{stacks-project} establishes 
that it is an integral (and hence irreducible) scheme. In particular, $S^{[n]}_\mathrm{red}$ is a quasi-projective, integral scheme over $\kk$, so it is a variety over $\kk$.

To see that $\Sym^n(S)$ has only canonical singularities, we apply the argument from the proof of \cite[Proposition~2.4, c.f.~(2.5)]{Beauville83}, replacing the rational 2-form throughout by a rational $2n$-form. Since each 
  resolution of $S^{[n]}_\mathrm{red}$ also provides a resolution of $\Sym^n(S)$, there cannot exist an exceptional divisor of negative discrepancy
  because $\Sym^n(S)$ has canonical singularities. It follows that 
  $S^{[n]}_\mathrm{red}$ also has canonical singularities.
\end{proof}
 
 Since irreducibility does not depend on the choice of scheme structure, the penultimate paragraph in the proof of Theorem~\ref{thm:HilbnSintro} above implies that the Hilbert scheme of $n$-points $S^{[n]}:=\Hilb^{n}(S)$ is irreducible without making reference to the work of Zheng~\cite{Zheng23}. In addition, our approach allows for a new proof of Fogarty's theorem:

\begin{theorem}[Fogarty]
\label{thm:fogarty}
Let $n\geq 1$. If $S$ is nonsingular, then $\Hilb^{n}(S)$ is nonsingular.
\end{theorem}
\begin{proof}
Note that $\Hilb^{n}(S)$ is nonsingular at $[Z]$ if and only if the complete local ring $\widehat{\mathcal{O}}_{\Hilb^{n}(S),[Z]}$ is regular for each $[Z]\in \Hilb^{n}(S)$ by \cite[Tag 07NU]{stacks-project}. The argument from the proof of Theorem~\ref{thm:HilbnSintro} above, combined with the first two sentences from the proof of Lemma~\ref{lem:completionKleinian} (which does not require the underlying reduced scheme structure), shows that this holds if the subgroup $\Gamma_k\subset \SL(2,\kk)$ associated to each point $p_k\in S$ is trivial, which is the case if $S$ is nonsingular.
\end{proof}

 \begin{proof}[Proof of Theorem~\ref{thm:nleq7intro}]
   For a connected and quasiprojective surface $S$ with canonical singularities, our goal is to show that $S^{[n]}:= \Hilb^n(S)$ is reduced for $n\leq 7$. In fact, we'll prove it's normal. Theorem~\ref{thm:reduced} shows that $\Hilb^n(\mathbb{A}^2/\Gamma)$ is reduced for $n\leq 7$. Therefore, in the proof of Theorem~\ref{thm:HilbnSintro} above, we may apply both Corollary~\ref{cor:completelocal} and Lemma~\ref{lem:completionKleinian} without having to take the underlying reduced scheme structure at any stage. We conclude that $S^{[n]}$ is normal, and hence reduced.
 \end{proof}

        \subsection{On surfaces with symplectic singularities}
      It remains to prove Corollary~\ref{cor:Beauville}. For this, let $\kk=\mathbb{C}$, the field of complex numbers. Recall that a variety $X$ over $\mathbb{C}$ has \emph{symplectic singularities} if the nonsingular locus $X_{\mathrm{reg}}$ has an everywhere nondegenerate closed 2-form $\omega$, and for any resolution of singularities $\pi\colon Y\to X$, the form $\pi^*(\omega)$ extends to a regular 2-form on $Y$. 

      \begin{proof}[Proof of Corollary~\ref{cor:Beauville}]
      We claim that the regular locus $U$ of $\Hilb^{n}(S)$ admits a symplectic form. Indeed, since the regular locus $S^{\mathrm{reg}}$ in $S$ is a smooth symplectic surface,  the open subset $V:= \Hilb^{n} (S^{\mathrm{reg}})$ of $U$ admits a symplectic form $\omega$ by Beauville~\cite[Proposition~5]{Beauville83}.  Since $U$ is nonsingular, it suffices to show that the complement of $V$ in $U$ has codimension two or more. 
  
   If $S$ is nonsingular, then $V=U$ and we're done. Otherwise, let $\pi\colon \Hilb^{n}(S)\to \Sym^n(S)$
  be the Hilbert--Chow morphism, and let $W\subset \Sym^n(S)$ be the locus consisting of points for which the support of the corresponding cycle in $S$ contains at least one singular point of $S$. Then $\pi^{-1}(W)$ has codimension two by Lemma \ref{lem} below, so the locus
 $U\setminus V=U\cap \pi^{-1}(W)$ has codimension at least two in $U$. The claim now follows. Finally, since each resolution of $\Hilb^{n}(S)$ also provides a resolution of $\Sym^n(S)$, the fact that the pullback of the symplectic form on $U$ extends to a regular 2-form on any given resolution follows from the fact that $\Sym^n(S)$ has symplectic singularities.
\end{proof}

 Let $p_1,\dots,p_\ell\in S$ denote the singular points. Since $\Sym^n(S)$ is a symplectic variety, it admits a stratification by symplectic leaves by Kaledin~\cite[Theorem 2.3]{Kaledin06}. In this case, if we define the set 
 \[
  \Lambda:=\Big\{(m_1, \dots,m_\ell; \nu) \mid m_1, \dots, m_\ell\in \ZZ_{\geq 0}, \nu \text{ is a partition of } d:= n-\sum_{1\leq k\leq \ell} m_k\geq 0\Big\},
  \]
  then the stratum associated to $\lambda:=(m_1, \dots,m_\ell; \nu) \in \Lambda$ is the locally closed subset 
      \[
  \Sym_\lambda^n(S):= \left\{ \sum_{k=1}^{\ell+l} m_ip_i \in \Sym^n(S) \mid \begin{array}{c} \text{partition }\nu=(m_{\ell+1}, \dots, m_{\ell+l}) \text{ has length }l(\nu)=l\\ p_{\ell+1}, \dots, p_{\ell+l}\in S^{\mathrm{reg}}\text{ satisfy }p_{\ell+i}\neq p_{\ell+j} \text{ for }i\neq j\end{array}\right\}
  \]
 of $\Sym^n(S)$.  Note that the dimension of the stratum $\Sym_\lambda^n(S)$ is $2l(\nu)$, and the open stratum is indexed by the element $\lambda=(0,\dots, 0;1, \dots, 1)$. 
  
\begin{lemma}\label{lem}
 Assume that $S$ is singular. Then $\pi^{-1}(W)$ has codimension two in $\Hilb^{n}(S)$.
\end{lemma}

\begin{proof}
 Consider a point $\sigma:=\sum_{1\leq k\leq \ell+l} m_ip_i \in \Sym_\lambda^n(S)$ for some  $\lambda:=(m_1, \dots,m_\ell; \nu) \in \Lambda$, where we write $\nu=(m_{\ell+1}, \dots, m_{\ell+l})$ as above. Note that $\sigma\in W$ if and only if $m_k\geq 1$ for some $1\leq k\leq \ell$. Therefore $W$ is the complement in $\Sym^n(S)$ of $\Sym^n(S^{\mathrm{reg}})$, so $W$ has codimension at least two. In fact, $W$ contains the codimension two strata indexed by those elements $\lambda=(0,\dots,0,1,0,\dots,0;\nu)$ with $l(\nu)=n-1$, so $W$ has ($\ell$ irreducible components of) codimension two. 

 For $1\leq k\leq \ell+l$, write $\pi_k\colon \Hilb^{m_k}(S)\to \Sym^{m_k}(S)$ for the Hilbert--Chow morphism. Apply both Lemma~\ref{lem:localpunctual} and Proposition~\ref{prop:formalnbd}, to obtain an isomorphism 
 \begin{equation}
 \label{eqn:productoffibres}
 \pi^{-1}(\sigma)\cong \prod_{1\leq k\leq \ell+l} \pi_k^{-1}(m_kp_k),
 \end{equation}
 where each fibre $\pi_k^{-1}(m_kp_k)$ depends only on $m_k\in \ZZ_{\geq 0}$ and the formal neighbourhood $\widehat{S}_{p_k}$. As in the proof of Theorem~\ref{thm:HilbnSintro} above, there is an isomorphism $\widehat{S}_{p_k}\cong \widehat{\mathbb{A}^2/\Gamma_k}$ for a finite subgroup $\Gamma_k\subset \SL(2,\kk)$, where $\Gamma_k$ is nontrivial for $1\leq k\leq \ell$, and it is trivial for $\ell+1\leq k\leq \ell+l$. In particular, $\pi_k^{-1}(m_kp_k)$ is isomorphic to the fibre of the Hilbert Chow morphism 
 \[
 \tau_k\colon \Hilb^{m_k}(\mathbb{A}^2/\Gamma_k)\longrightarrow \Sym^{m_k}(\mathbb{A}^2/\Gamma_k)
 \] 
 over a point in the unique minimal symplectic leaf. This morphism is obtained by variation of GIT quotient for a Nakajima quiver variety: for $1\leq k\leq \ell$, this follows from Theorem~\ref{thm:Hilbnintro} and for $\ell+1\leq k\leq \ell+l$, it is well known \cite{Craw21, NakajimaHilbert}. Therefore, each  $\tau_k$ is semismall by \cite[Theorem~A.1]{BC20}. 
 
 Finally, let $\sigma\in W$ lie in a leaf $\Sym_\lambda^n(S)$ for  $\lambda=(0,\dots,0,1,0,\dots,0;\nu)$ with $l(\nu)=n-1$.
  The only nonzero values of $m_k$ satisfy $m_k=1$, so the corresponding maps $\tau_k$ are isomorphisms. Therefore each $\pi_k^{-1}(p_k)$ is a singleton, and hence so is $\pi^{-1}(\sigma)$ by \eqref{eqn:productoffibres}. Thus, $\pi$ is an isomorphism over each symplectic leaf in $W$ that has codimension two in $\Sym^n(S)$. The symplectic leaves in $W$ that are not of codimension two in $\Sym^n(S)$ are of codimension at least four, but then semismallness implies that the exceptional locus of $\pi$ over $W$ does not contain any divisor.
\end{proof}

\appendix

\section{The case when vertex $0$ does not lie in $J$}
\label{sec:appendix}
 As we observed at the start of Section~\ref{sec:cornering}, the proofs of Theorems~\ref{thm:HilbnSintro}, \ref{thm:nleq7intro}, \ref{thm:Hilbnintro} and Corollary~\ref{cor:Beauville} require only the case $J=\{0\}$, so the reader interested only in those results can ignore this appendix; in fact, the same is true for any subset $J\subseteq \{0,1,\dots, r\}$ that contains vertex $0$. 
   
 However, the proofs of Theorems~\ref{thm:closedimmersion} and \ref{thm:CGGS21} for subsets $J$ satisfying $0\not\in J$ rely on the second paragraph in the proof of Lemma~\ref{lem:nocyclesneededinA}, and that requires the following technical result.

\begin{proposition}
\label{prop:CBresultsforvertexj}
  Let $n\geq 1$. For any $0\leq j\leq r$, the algebra $\kk[\Rep(\PiMcKay ,n\delta)]^{\GL(n\delta)}$ is generated by the trace functions for cycles in $\McKay$ with head and tail at vertex $j\in \McKay_0$.
\end{proposition}

 When the vertex of interest is $j=0$, we deduced this result in the proof of Proposition~\ref{prop:M0Sym^n} from the work of Crawley-Boevey~\cite[Lemma~3.3 and proof of Theorem~3.4]{CBdecomp}.
 
 Before providing a proof for any $0\leq j\leq r$, we present a well-known result for which we could not find a reference. Recall from \eqref{eqn:ADEhypersurface} that for $R_0=\kk[x,y]^{\Gamma}$, the preprojective algebra $\PiMcKay$ for the McKay quiver $\McKay$ is isomorphic to the endomorphism algebra of the $R_0$-module $\bigoplus_{0\leq i\leq r} R_i$.

\begin{lemma}
\label{lem:TjDeltaj}
 For each $0\leq j\leq r$, there exist $R_0$-module homomorphisms $\Delta_j\colon R_0 \to \End_{R_0}(R_j)$ and $T_j\colon \End_{R_0}(R_j)\to R_0$ whose composition $T_j\circ \Delta_j\colon R_0\to R_0$ is multiplication by $\delta_j:=\dim \rho_i$.  In particular, $\Delta_j$ is injective and $T_j$ is surjective.
\end{lemma}
\begin{proof}
 Note that $R_0\cong H^0(\mathcal{O}_{\mathbb{A}^2/\Gamma})\cong H^0(\mathcal{O}_S)$ where $\tau\colon S\to \mathbb{A}^2/\Gamma$ is the minimal resolution. In addition, regard $S$ as the fine moduli space $\Rep(\PiMcKay,\delta)\git_\theta \GL(\delta)$ of $\theta$-stable $\PiMcKay$-modules of dimension vector $\delta$ following Crawley--Boevey~\cite{CBMcKay}, where $\theta\in \GL(\delta)^\vee$ satisfies $\theta_k>0$ for $k\neq 0$. This endows $S$ with a tautological vector bundle $\bigoplus_{0\leq k\leq r} \mathcal{R}_k^{\oplus \delta_k}$ satisfying $\mathcal{R}_0\cong \mathcal{O}_S$ and $\rank \mathcal{R}_k = \dim\rho_k$ for $k\neq 0$, whose endomorphism algebra is isomorphic to the skew group algebra of $\Gamma$ (see \cite[Lemma~4.1]{CIK17}). Applying the Morita equivalence argument as in \cite[proof of Lemma~3.1]{Craw21} establishes an isomorphism 
  \[
  \PiMcKay\cong \End_{\mathcal{O}_S}\bigg( \bigoplus_{0\leq k\leq r} \mathcal{R}_k\bigg)
  \]
 of $\kk$-algebras.  Combining this with the isomorphism \eqref{eqn:PiEnd} gives $\End_{R_0}(R_j)\cong e_j\PiMcKay e_j\cong \End_{\mathcal{O}_S}(\mathcal{R}_j)$.
  
 We construct the maps $\Delta_j$ and $T_j$ locally on $S$. Choose an affine cover $\{U_k\}$ of $S$ that trivialises $\mathcal{R}_j$, so $\mathcal{R}_j\vert_{U_k}\cong \mathcal{O}_{U_k}^{\oplus \delta_j}$ for each index $k$. Over any such affine chart, we have
  \[ 
\End_{\mathcal{O}_{U_k}}\!\big(\mathcal{R}_j\vert_{U_k}\big)\cong H^0(\mathcal{O}_{U_k})\otimes_\kk \Mat(\delta_j\times \delta_j),
  \]
  so $\End_{\mathcal{O}_S}(\mathcal{R}_j)$ is locally an algebra of $H^0(\mathcal{O}_{U_k})$-valued square matrices. Now define
  \[
  \Delta_k\vert_{U_k}\colon H^0(\mathcal{O}_{U_k}) \longrightarrow \End_{\mathcal{O}_{U_k}}\!\big(\mathcal{R}_j\vert_{U_k}\big)
  \]
  by sending $g$ to the diagonal matrix $g\otimes \id_{\delta_j\times \delta_j}$. These maps all glue to give a well-defined map $\Delta_j\colon R_0\cong H^0(\mathcal{O}_S)\to \End_{\mathcal{O}_S}(\mathcal{R}_j)\cong \End_{R_0}(R_j)$. Similarly, define
  \[
  T_j\vert_{U_k}\colon \End_{\mathcal{O}_{U_k}}\!\big(\mathcal{R}_j\vert_{U_k}\big) \longrightarrow 
  H^0(\mathcal{O}_{U_k})
  \]
  by sending each $H^0(\mathcal{O}_{U_k})$-valued $\delta_j\times \delta_j$ matrix to its trace. These maps glue because trace is invariant under conjugation, giving a map $T_j\colon \End_{R_0}(R_j)\cong \End_{\mathcal{O}_{S}}(\mathcal{R}_j)\to H^0(\mathcal{O}_S)\cong R_0$.
  
  Finally, let $g\in R_0\in H^0(\mathcal{O}_S)$ and $s\in S$. Then $s\in U_k$ for some index $k$,  and we compute
  \begin{equation}
      \label{eqn:TjDeltaJ}
  T_j\big(\Delta_j(g)\big)(s) = \big(T_j\vert_{U_k}\circ \Delta_j\vert_{U_k}\big)(g\vert_{U_k})(s) = \tr\big(g\vert_{U_k}\otimes \id_{\delta_j\times \delta_j}\big)(s) = \delta_j g(s),
   \end{equation}
  so $(T_j\circ \Delta_j)(g)=\delta_jg$ as required.
\end{proof}

\begin{proof}[Proof of Proposition~\ref{prop:CBresultsforvertexj}]
 First, assume $n=1$. Identify the functions $p_1, p_2, p_3$ in $R_0$ corresponding to the chosen algebra generators from \eqref{eqn:ADEhypersurface} with the corresponding linear combinations of classes of paths in $e_0 \PiMcKay e_0$. For $1\leq \ell\leq 3$, we have $p_\ell=\tr_{p_\ell}$ because $\dim \rho_0=1$, and 
 \begin{equation}
     \label{eqn:qell}
  q_\ell:= \Delta_j(p_\ell)\in \End_{R_0}(R_j)\cong e_j \PiMcKay e_j
 \end{equation}
 is a linear combination of finitely many cycles in $\McKay$. Since trace is linear, it suffices to show that $\kk[\Rep(\PiMcKay,\delta)]^{\GL(\delta)}\cong R_0\cong H^0(\mathcal{O}_S)$ is generated by $\tr_{q_1}, \tr_{q_2}, \tr_{q_3}$. As in the proof of Lemma~\ref{lem:TjDeltaj}, any point $s\in S$ lies in some chart $U_k$, and moreover, may be regarded as the $\GL(\delta)$-orbit of a $\theta$-stable point $B\in \Rep(\PiMcKay,\delta)$. Now \eqref{eqn:TjDeltaJ} and the fact that $p_\ell=\tr_{p_\ell}$ from above give
 \begin{equation}
     \label{eqn:qkpk}
 \tr_{q_\ell}(B) = T_j(q_\ell)(B) = T_j\big(\Delta_j(p_\ell)\big)(B) = (\delta_j p_\ell)(B) = \delta_j\tr_{p_\ell}(B)
 \end{equation}
 for $1\leq \ell\leq 3$, so the scaled function $\frac{1}{\delta_j}\tr_{q_\ell}$ equals $\tr_{p_\ell}$. These latter three trace functions generate $\kk[x,y]^\Gamma\cong\kk[\Rep(\PiMcKay,\delta)]^{\GL(\delta)}$ by the isomorphism $e_0\PiMcKay e_0\cong R_0$ from \cite[Theorem~8.10]{CBHolland}, and hence so do the trace functions $\tr_{q_\ell}$ for $1\leq \ell\leq 3$ coming from cycles through vertex $j$ as required.
   
   Now consider $n> 1$. The isomorphism $\nu_2$ from \eqref{eqn:nu_2} is obtained from a $\kk$-algebra isomorphism 
 \begin{equation}
     \label{eqn:CBisomorphism}
 f^*\colon \kk[\Rep(\PiMcKay,n\delta)]^{\GL(n\delta)}\longrightarrow \Big(\big(\kk[\Rep(\PiMcKay,\delta)]^{\GL(\delta)}\big)^{\otimes n}\Big)^{\mathfrak{S}_n},
  \end{equation}
 where the symmetric group $\mathfrak{S}_n$ acts by permuting the $n$ factors of the tensor algebra. The previous paragraph establishes that $\kk[\Rep(\PiMcKay,\delta)]^{\GL(\delta)}$ is generated by the trace functions $\tr_{q_\ell}$ for $1\leq \ell\leq 3$, where each $q_\ell$ is a linear combination of cycles in the McKay quiver $\McKay$ passing through vertex $j$. For any $1\leq k\leq n$, let $\pi_k\colon (\mathbb{A}^2/\Gamma)^n\to \mathbb{A}^2/\Gamma$ denote the projection to the $k^{th}$ factor. The proof of \cite[Theorem~3.4]{CBdecomp} shows that the elements
\[
s'_{r_1, r_2, r_3}=\sum_{1\leq k\leq n}\big(\tr_{q_1}^{r_1}\tr_{q_2}^{r_2}\tr_{q_3}^{r_3}\big)\circ \pi_k
\]
 generate $\big(\big(\kk[\Rep(\PiMcKay,\delta)]^{\GL(\delta)}\big)^{\otimes n}\big)^{\mathfrak{S}_n}$, where $r_1, r_2, r_3\geq 0$. Since $\tr_{q_k}=\delta_j \tr_{p_k}$ by \eqref{eqn:qkpk}, we obtain
  \[
 \tr_{q_k}\tr_{q_\ell} = \delta_j^2\tr_{p_k}\tr_{p_\ell} = \delta_j^2\tr_{p_k p_\ell} = \delta_j\tr_{q_k q_\ell},
 \]
 where the second equality follows from $\dim \rho_0=1$, and the third equality follows as in \eqref{eqn:qkpk}. Then $\tr_{q_1}^{r_1}\tr_{q_2}^{r_2}\tr_{q_3}^{r_3} = \delta_j^{(r_1+r_2+r_3)-1}\tr_q$ for the class $q:=q_1^{r_1}q_2^{r_2}q_3^{r_3}\in e_j\PiMcKay e_j$ of cycles through $j$, so
 \[
 s'_{r_1, r_2, r_3}= \delta_j^{(r_1+r_2+r_3)-1}\sum_{1\leq k\leq n}\tr_{q}\circ \pi_k.
 \]
 As the proof of \cite[Theorem~3.4]{CBdecomp} shows, this is the image under the isomorphism $f^*$ from \eqref{eqn:CBisomorphism} of the scalar $\delta_j^{(r_1+r_2+r_3)-1}$ times the trace function $\tr_q$ in $\kk[\Rep(\PiMcKay,n\delta)]$. Therefore, the functions $\tr_q$ for classes of the form $q = q_1^{r_1}q_2^{r_2}q_3^{r_3}$ in $\McKay$ generate $\kk[\Rep(\PiMcKay,n\delta)]^{\GL(n\delta)}$. The result follows since each $q_1, q_2, q_3$, and hence every such class $q$, is  a linear combination of cycles through vertex $j$.
\end{proof}

 \small{
\bibliographystyle{plain}

}

   \end{document}